\documentclass[12pt,reqno]{amsart}

\usepackage{multirow}
\usepackage{hhline}
\usepackage{float}

\usepackage{mathtools}
\usepackage{times}
\usepackage[T1]{fontenc}
\usepackage{mathrsfs}
\usepackage{latexsym}
\usepackage[dvips]{graphics}
\usepackage[titletoc, title]{appendix}
\usepackage{amsmath,amsfonts,amsthm,amssymb,amscd}
\usepackage[dvipsnames]{xcolor}
\usepackage{hyperref}
\usepackage{amsmath}
\usepackage[utf8]{inputenc}

\usepackage{color}
\usepackage{breakurl}

\usepackage{comment}
\newcommand{\bburl}[1]{\textcolor{blue}{\url{#1}}}
\newcommand{\seqnum}[1]{\href{https://oeis.org/#1}{\rm \underline{#1}}}

\usepackage{caption}

\newtheorem{thm}{Theorem}[section]

\newtheorem{cor}[thm]{Corollary}

\newtheorem{lem}[thm]{Lemma}

\newtheorem{rek}[thm]{Remark}
\newtheorem{prob}[thm]{Problem}

\usepackage[utf8]{inputenc}

\DeclareFixedFont{\ttb}{T1}{txtt}{bx}{n}{12} 
\DeclareFixedFont{\ttm}{T1}{txtt}{m}{n}{12}  

\usepackage{color}
\definecolor{deepblue}{rgb}{0,0,0.5}
\definecolor{deepred}{rgb}{0.6,0,0}
\definecolor{deepgreen}{rgb}{0,0.5,0}

\usepackage{listings}

\newcommand\pythonstyle{\lstset{
language=Python,
basicstyle=\ttm,
morekeywords={self},              
keywordstyle=\ttb\color{deepblue},
emph={MyClass,__init__},          
emphstyle=\ttb\color{deepred},    
stringstyle=\color{deepgreen},
frame=tb,                         
showstringspaces=false
}}

\lstnewenvironment{python}[1][]
{
\pythonstyle
\lstset{#1}
}
{}

\newcommand\pythoninline[1]{{\pythonstyle\lstinline!#1!}}

\definecolor{ao}{rgb}{0.0, 0.5, 0.0}

\numberwithin{equation}{section}

\DeclareFontFamily{U}{mathx}{}
\DeclareFontShape{U}{mathx}{m}{n}{<-> mathx10}{}
\DeclareSymbolFont{mathx}{U}{mathx}{m}{n}
\DeclareMathAccent{\widehat}{0}{mathx}{"70}
\DeclareMathAccent{\widecheck}{0}{mathx}{"71}

\title{A Pair of Diophantine Equations and Fibonacci-Like Sequences}

\author[H. V. Chu]{H\`ung Vi\d{\^e}t Chu}
\email{\textcolor{blue}{\href{mailto:hchu@wlu.edu}{hchu@wlu.edu}}}
\address{Department of Mathematics, Washington and Lee University, Lexington, VA 24450, USA}   

\author[R. Gulecha]{Rishabh Gulecha}
\email{\textcolor{blue}{\href{mailto:rishabhg@tamu.edu}{rishabhg@tamu.edu}}}
\address{Department of Mathematics, Texas A\&M University, College Station, TX 77843, USA}

\author[S. Guo]{Sicheng Guo}
\email{\textcolor{blue}{\href{mailto:sophiasg@umich.edu}{sophiasg@umich.edu}}}
\address{Department of Mathematics, University of Michigan, Ann Arbor, MI 48109, USA}   

\author[N. Johnson]{Nathanael Johnson}
\email{\textcolor{blue}{\href{mailto:nathan.erikson.9701@gmail.com}{nathan.erikson.9701@gmail.com}}}
\address{Department of Mathematics, The Ohio State University, Columbus, OH 43210, USA}

\author[S. J. Miller]{Steven J. Miller}
\email{\textcolor{blue}{\href{mailto:sjm1@williams.edu}{sjm1@williams.edu},
\href{mailto:Steven.Miller.MC.96@aya.yale.edu}{Steven.Miller.MC.96@aya.yale.edu}}}
\address{Department of Mathematics and Statistics, Williams College, Williamstown, MA 01267, USA}

\author[Y. Shin]{Yeju Shin}
\email{\textcolor{blue}{\href{mailto:yejushin@umass.edu}{yejushin@umass.edu}}}
\address{Department of Mathematics and Statistics, University of Massachusetts Amherst, Amherst, MA 01003, USA}

\begin{document}

\thanks{This work was partially supported by the National Science Foundation DMS2341670. We thank the
participants at Polymath Jr. 2025 for helpful discussions.}



\maketitle

\begin{abstract}
Given two relatively prime numbers $a$ and $b$, it is known that exactly one of the two Diophantine equations has a nonnegative integral solution $(x,y)$:
$$
    ax + by \ =\ \frac{(a-1)(b-1)}{2}\quad \mbox{ and }\quad 1 + ax + by \ =\ \frac{(a-1)(b-1)}{2}.
$$
Furthermore, the solution is unique. This paper surveys recent results on finding the solution and determining which equation is used when $a$ and $b$ are taken from certain sequences. We contribute to the literature by finding $(x,y)$ when $a$ and $b$ are consecutive terms of sequences having the Fibonacci recurrence and arbitrary initial terms. 
\end{abstract}

\tableofcontents

\section{Introduction}\label{intro}
In the study of cyclotomic polynomials $\Phi_{pq}(x)$ for prime numbers $p$ and $q$, Beiter \cite{Be} used the result that exactly one of the two equations
$$px + qy\ =\ \frac{(p-1)(q-1)}{2}\quad \mbox{ and }\quad 1 + px + qy\ =\ \frac{(p-1)(q-1)}{2}$$
has a nonnegative integral solution $(x,y)$, and the solution is unique. In general, the same conclusion holds for every pair $(a,b)\in \mathbb{N}^2$ with $\gcd(a,b) = 1$.

\begin{thm}\cite[Theorem 1.1]{C4}\label{ot}
For relatively prime $a, b\in \mathbb{N}$, exactly one of the following equations has a nonnegative, integral solution $(x,y)$:
\begin{align}
    \label{e1}ax + by &\ =\ \frac{(a-1)(b-1)}{2},\\
    \label{e2}1 + ax + by &\ =\ \frac{(a-1)(b-1)}{2}.
\end{align}
Furthermore, the solution is unique. 
\end{thm}

\begin{proof} Let $(a,b)\in \mathbb{N}^2$ with $\gcd(a,b) = 1$.
We start with an easy observation: if the equation $ax + by = n$ has a solution 
$(x,y) = (r,s)\in \mathbb{Z}^2$ with $r < b$ and $s < 0$, then the equation has no nonnegative integral solutions. This is because given $ar + bs = n$, all integral solutions of $ax + by = n$ must have the form $(x,y) = (r + tb, s - ta)$, where $t$ is an integer. In order that $r +tb \ge 0$, we need $t\ge 0$, which implies that $s - ta\le s < 0$. 

Let $k = (a-1)(b-1)/2$. Since $\gcd(a,b) = 1$, there are unique integers $x_1$ and $x_2$ in $[0, b-1]$ with $ax_1\equiv k\mod b$ and $ax_2\equiv (k-1)\mod b$. Let $y_1 := (k-ax_1)/b$ and $y_2 := (k-1-ax_2)/b$. We have
$$a(x_1 + x_2)\ \equiv \ 2k-1\ \equiv\ -a\mod b\ \Longrightarrow\  b|(x_1 + x_2 + 1).$$
It follows from $1\le x_1 + x_2 + 1\le 2b-1$ that $x_1 + x_2  = b-1$. Hence,
$$y_1 + y_2\ =\ \frac{(k-ax_1) + (k-1-ax_2)}{b}\ =\ \frac{2k-1-a(b-1)}{b}\ =\ -1.$$
Since $y_1$ and $y_2$ are integers, exactly one of them is nonnegative. Combined with the observation at the beginning, this proves  that exactly one equation has nonnegative integral solutions. 

Uniqueness follows from the standard proof by contradiction.
\end{proof}

Given $(a,b)\in \mathbb{N}^2$ with $\gcd(a,b) = 1$, we denote the solution $(x,y)$ of \eqref{e1} (if it exists) by $$(x^{(0)}(a,b),\ y^{(0)}(a,b))$$ and denote the solution of \eqref{e2} (if it exists) by $$(x^{(1)}(a,b),\ y^{(1)}(a,b)).$$

Besides discussing recent progress in the study of Equations \eqref{e1} and \eqref{e2} from \cite{ACLLMM, CCKKLMYY, C4, Da}, we contribute to the literature by finding solutions to \eqref{e1} and \eqref{e2} when $a$ and $b$ are consecutive terms of sequences that have the Fibonacci recurrence with arbitrary initial terms. We call these sequences \textit{Fibonacci-like}. Additionally, our investigation of Fibonacci-like sequences reveals why existing formulas in the literature are often given in $6$ cases (see \cite[Theorems 1.4 and 1.6]{C4} and \cite[Corollary 1.4]{CCKKLMYY}, for example). Briefly speaking, $F_n$ is even if and only if $3$ divides $n$; since the solution $(x,y)$ depends not only the parity of $F_n$ but also on the parity of $n$ (due to the Cassini's identity), we need to consider $6$ cases.  Along the way, we suggest various problems for future investigations. 

The paper is structured as follows: Section \ref{explicitformula} presents numerous identities inspired by the two Diophantne equations; Section \ref{Flikeseq} proves new formulas that compute the solution $(x,y)$ when $a$ and $b$ are  consecutive terms of Fibonacci-like sequences; Section \ref{whicheq} shows a method and its applications in determining which equation has a nonnegative integral solution.

\section{Identities from the two Diophantine equations}\label{explicitformula}

This section summarizes the ongoing study of the solution to \eqref{e1} and \eqref{e2} when $a$ and $b$ are relatively prime numbers from a well-known sequence. A typical result gives explicit formulas to compute the solution as $a$ and $b$ run along the sequence.

\subsection{The Fibonacci sequence}

Let $(F_n)_{n=0}^\infty$ be the Fibonacci sequence with $F_0 = 0$, $F_1 = 1$, and $F_n = F_{n-1} + F_{n-2}$ for $n\ge 2$. For each positive integer $n\ge 2$, 
$$\gcd(F_n, F_{n+1})\ =\ \gcd(F_n, F_{n-1} + F_{n}) \ =\ \gcd(F_{n-1}, F_n),$$
so for every $n\in \mathbb{N}$,
$$\gcd(F_n, F_{n+1})\ =\ \gcd(F_1, F_2)\ =\ 1;$$
that is, consecutive terms of the Fibonacci sequence are relatively prime. This fact and Theorem \ref{ot} inspire the following identities where we consider $6$ cases as discussed in Section \ref{intro}.

\begin{thm}\cite[Theorem 1.4]{C4}\label{ori}
    For $k\ge 1$, the following hold:
    \begin{align*}
        \frac{1}{2}(F_{6k-1}-1)F_{6k} + \frac{1}{2}(F_{6k-1} - 1)F_{6k+1}&\ =\ \frac{(F_{6k}-1)(F_{6k+1}-1)}{2};\\
        \frac{1}{2}(F_{6k+1} - 1)F_{6k+1} + \frac{1}{2}(F_{6k-1}-1)F_{6k+2}&\ =\ \frac{(F_{6k+1}-1)(F_{6k+2}-1)}{2};\\
        \frac{1}{2}(F_{6k+1}-1)F_{6k+2} + \frac{1}{2}(F_{6k+1}-1)F_{6k+3}&\ =\ \frac{(F_{6k+2}-1)(F_{6k+3}-1)}{2};\\ 
        1 + \frac{1}{2}(F_{6k+2}-1)F_{6k+3} + \frac{1}{2}(F_{6k+2}-1)F_{6k+4}&\ =\ \frac{(F_{6k+3}-1)(F_{6k+4}-1)}{2};\\
        1 + \frac{1}{2}(F_{6k+4}-1)F_{6k+4} + \frac{1}{2}(F_{6k+2}-1)F_{6k+5}&\ =\ \frac{(F_{6k+4}-1)(F_{6k+5}-1)}{2};\\
        1 + \frac{1}{2}(F_{6k+4}-1)F_{6k+5} + \frac{1}{2}(F_{6k+4}-1)F_{6k+6}&\ =\ \frac{(F_{6k+5}-1)(F_{6k+6}-1)}{2}.
    \end{align*}
\end{thm}

Similarly, since for each $n\in \mathbb{N}$,
$$\gcd(F_n, F_{n+2})\ =\ \gcd(F_n, F_n + F_{n+1})\ =\ \gcd(F_n, F_{n+1})\ =\ 1,$$
we may set $(a,b) = (F_n, F_{n+2})$ to obtain the next set of identities.

\begin{thm}\cite[Theorem 1.6]{C4}
    For $k\ge 0$, the following hold:
    \begin{align*}
        \frac{1}{2}(F_{6k+2}-1)F_{6k+1} + \frac{1}{2}(F_{6k-1}-1)F_{6k+3}&\ =\ \frac{(F_{6k+1}-1)(F_{6k+3}-1)}{2};\\
        \frac{1}{2}(F_{6k+2}-1)F_{6k+2} + \frac{1}{2}(F_{6k+1}-1)F_{6k+4}&\ =\ \frac{(F_{6k+2}-1)(F_{6k+4}-1)}{2};\\
        \frac{1}{2}(F_{6k+4}-1)F_{6k+3} + \frac{1}{2}(F_{6k+1}-1)F_{6k+5}&\ =\ \frac{(F_{6k+3}-1)(F_{6k+5}-1)}{2};\\
        1 + \frac{F_{6k+5}-1}{2}F_{6k+4} + \frac{1}{2}(F_{6k+2}-1)F_{6k+6}&\ =\ \frac{(F_{6k+4}-1)(F_{6k+6}-1)}{2};\\
         1 + \frac{F_{6k+5}-1}{2}F_{6k+5} + \frac{1}{2}(F_{6k+4}-1)F_{6k+7}&\ =\ \frac{(F_{6k+5}-1)(F_{6k+7}-1)}{2};\\
         1 + \frac{F_{6k+1}-1}{2}F_{6k} + \frac{1}{2}(F_{6k-2}-1)F_{6k+2}&\ =\ \frac{(F_{6k}-1)(F_{6k+2}-1)}{2}.
    \end{align*}
\end{thm}

In Section \ref{Flikeseq}, we establish formulas for the solution when $a$ and $b$ are consecutive terms of general Fibonacci-like sequences.

\subsection{Fibonacci numbers squared and cubed}

More intriguing solutions occur when we let $(a,b) = (F_n^2, F_{n+1}^2)$ or $(F_n^3, F_{n+1}^3)$. 

\begin{thm}\cite[Corollary 1.4]{CCKKLMYY} Let $n$ be a positive integer at least $2$. The following are true.
\begin{enumerate}
    \item If $n\equiv 0, 2, 3, 5\mod 6$, 
    \begin{equation}\label{e3}(x^{(0)}(F_n^2, F_{n+1}^2),\ y^{(0)}(F_n^2, F_{n+1}^2))\ =\ \left(F_n^2 - \frac{F^2_{n-1}+1}{2}, \frac{F^2_{n-1}-1}{2}\right).\end{equation}
    \item If $n\equiv 1\mod 6$,
     \begin{equation}(x^{(1)}(F_n^2, F_{n+1}^2),\ y^{(1)}(F_n^2, F_{n+1}^2))\ =\ \left(\frac{F^2_{n}-3}{2}, \frac{F_n^2-F^2_{n-1}-1}{2}\right).\end{equation}
     \item If $n\equiv 4\mod 6$,
        \begin{equation}(x^{(1)}(F_n^2, F_{n+1}^2),\ y^{(1)}(F_n^2, F_{n+1}^2))\ =\ \left(\frac{F^2_{n}+1}{2}, \frac{F_n^2-F^2_{n-1}-1}{2}\right).\end{equation}
\end{enumerate}
\end{thm}

\begin{proof}
We prove \eqref{e3}. The other two are similar.  We have
$$F^2_n \ =\ (F_{n+1} - F_{n-1})^2\ =\ F_{n+1}^2 - 2F_{n+1}F_{n-1} + F_{n-1}^2,$$
so 
\begin{equation}\label{e4}F^2_n - F_{n+1}^2 - F_{n-1}^2 \ =\ -2F_{n+1}F_{n-1}.\end{equation}

By the Cassini's identity\footnote{For each integer $n$, $F_{n-1}F_{n+1}-F_n^2 = (-1)^n$.},
\begin{equation}\label{e5}(F_{n-1}F_{n+1}-F_n^2)^2\ =\ 1\ \Longrightarrow\ F^2_{n-1}F^2_{n+1}-2F_{n-1}F_n^2F_{n+1} + F_n^4\ =\ 1.\end{equation}

From \eqref{e4} and \eqref{e5}, we obtain
$$F^2_{n-1}F^2_{n+1}+F^2_n(F^2_n - F_{n+1}^2 - F_{n-1}^2) + F_n^4\ =\ 1.$$
Hence,
$$2F_n^4 + F^2_{n-1}F^2_{n+1} - F^2_{n-1}F_n^2 \ =\ 1 + F_{n}^2 F^2_{n+1}.$$
Adding $-F_n^2-F^2_{n+1}$ to both sides gives
$$(2F_n^2-F^2_{n-1}-1)F_n^2 + (F^2_{n-1}-1)F^2_{n+1} \ =\ (F_n^2-1)(F_{n+1}^2-1).$$
Therefore,
$$\left(F_n^2 - \frac{F^2_{n-1}+1}{2}\right)F^2_{n} + \frac{F^2_{n-1}-1}{2}F^2_{n+1}\ =\ \frac{(F^2_n-1)(F^2_{n+1}-1)}{2}.$$

For $n\ge 2$,
$$F_n^2 - \frac{F^2_{n-1}+1}{2}\ \ge\ 0\mbox{ and }F^2_{n-1}-1\ \ge\ 0.$$
Furthermore, $F_{n-1}$ is odd, so $(F^2_{n-1}+1)/2$ is an integer. 
\end{proof}

The solution when $(a,b) = (F_n^3, F_{n+1}^3)$ is even more interesting with (alternating) sums of Fibonacci numbers cubed. 

\begin{thm}\cite[Theorem 1.5]{CCKKLMYY} For $n\ge 2$, we have
\begin{align}
    \label{e6}\left(\sum_{i=1}^{2n-1}(-1)^{i-1}F_i^3\right)F^3_{2n-1} + \left(\sum_{i=2}^{2n-2} F_i^3\right)F^3_{2n}&\ =\ \frac{(F^3_{2n-1}-1)(F^3_{2n}-1)}{2};\\
    1+\left(\sum_{i=1}^{2n}(-1)^{i}F_i^3-1\right)F^3_{2n} + \left(\sum_{i=2}^{2n-1} F_i^3\right)F^3_{2n+1}&\ =\ \frac{(F^3_{2n}-1)(F^3_{2n+1}-1)}{2}.\nonumber
\end{align}
\end{thm}

\begin{proof}[Proof of \eqref{e6}]
We recall  \cite[Theorem 1]{Fr}, which gives a formula for the (alternating) sum of Fibonacci numbers cubed: for $m\ge 1$,
\begin{align*}
    \sum_{i=1}^m F_i^3&\ =\ \frac{1}{4}(F_{3m+3}+F_{3m}) - F^3_{m+1} - F_m^3 + \frac{1}{2};\\
    \sum_{i=1}^m (-1)^iF_i^3&\ =\ \frac{1}{4}((-1)^mF_{3m+3} + (-1)^{m+1}F_{3m}) - (-1)^m F^3_{m+1} - (-1)^{m+1}F_m^3 + \frac{1}{2}.
\end{align*}
Furthermore, the well-known identity $F_{3n} = 5F_n^3 + 3(-1)^nF_n$ gives
$$F_{6n}\ =\ 5F^3_{2n} + 3F_{2n}, F_{6n-3}\ =\ 5F^3_{2n-1}-3F_{2n-1}, \mbox{ and } F_{6n-6}\ =\ 5F^3_{2n-2}+3F_{2n-2}.$$
These identities allow us to rewrite the left side of \eqref{e6}, denoted by $T(n)$, as
\begin{align}
    T(n)\ =\ &\left(\frac{F_{6n}}{4} - \frac{F_{6n-3}}{4} - F^3_{2n} + F_{2n-1}^3 - \frac{1}{2}\right)F^3_{2n-1}+\nonumber\\
    &  \left(\frac{F_{6n-3}}{4} + \frac{F_{6n-6}}{4} - F^3_{2n-1} - F^3_{2n-2} - \frac{1}{2}\right)F^3_{2n}\nonumber\\
    \ =\ &\frac{F_{6n}-F_{6n-3}}{4}F^3_{2n-1} + F^6_{2n-1} + \frac{F_{6n-3} + F_{6n-6}}{4}F^3_{2n}-\nonumber\\
    &F^3_{2n-2}F^3_{2n} - 2F^3_{2n-1}F^3_{2n}-\frac{F^3_{2n-1} + F^3_{2n}}{2}\nonumber\\
    \ =\ &\frac{5}{4}F^3_{2n-1}F^3_{2n} + \frac{3}{4}F^3_{2n-1}F_{2n} - \frac{5}{4}F^6_{2n-1} + \frac{3}{4}F^4_{2n-1} + F^6_{2n-1} + \frac{5}{4}F^3_{2n-1}F^3_{2n} -\nonumber\\
    &\frac{3}{4}F_{2n-1}F^3_{2n} + \frac{5}{4}F^3_{2n-2}F^3_{2n} + \frac{3}{4}F_{2n-2}F^3_{2n} - F^3_{2n-2}F^3_{2n} - \nonumber\\
    &2F^3_{2n-1}F^3_{2n}-\frac{F^3_{2n-1} + F^3_{2n}}{2}\nonumber\\
    \ =\  &- \frac{1}{4}F^6_{2n-1} + \frac{1}{4}F^3_{2n-2}F^3_{2n} + \frac{3}{4}F_{2n-1}^3F_{2n} + \frac{3}{4}F^4_{2n-1}-\nonumber\\
    &\label{e7}\frac{3}{4}F_{2n-1}F^3_{2n} + \frac{3}{4}F_{2n-2}F^3_{2n} + \frac{F^3_{2n-1}F^3_{2n}-F^3_{2n-1}-F^3_{2n}}{2}.
\end{align} 

Cubing both sides of the Cassini's identity $F_{2n-2}F_{2n} = -1 + F^2_{2n-1}$, we have
\begin{equation}\label{e8}
F_{2n-2}^3F^3_{2n}\ =\ F^6_{2n-1} - 3F^4_{2n-1} + 3F^2_{2n-1} - 1;\end{equation}
additionally
\begin{equation}\label{e9}F_{2n-2}F^3_{2n}\ =\ F^2_{2n}F_{2n-2}F_{2n}\ =\ F^2_{2n}(F^2_{2n-1}-1).\end{equation}

From \eqref{e7}, \eqref{e8}, and \eqref{e9}, we obtain
\begin{align*}
T(n)\ =\ &- \frac{1}{4}F^6_{2n-1} + \frac{1}{4}\left(F^6_{2n-1} - 3F^4_{2n-1} + 3F^2_{2n-1} - 1\right) + \frac{3}{4}F_{2n-1}^3F_{2n} + \frac{3}{4}F^4_{2n-1}-\\
    &\frac{3}{4}F_{2n-1}F^3_{2n} + \frac{3}{4}(F^2_{2n-1}F^2_{2n} - F^2_{2n}) + \frac{F^3_{2n-1}F^3_{2n}-F^3_{2n-1}-F^3_{2n}}{2}\\
    \ =\ &\frac{3}{4}F^2_{2n-1} - \frac{1}{4} + \frac{3}{4}F_{2n-1}^3F_{2n} - \frac{3}{4}F_{2n-1}F^3_{2n} + \frac{3}{4}F_{2n-1}^2F_{2n}^2 - \frac{3}{4}F_{2n}^2 + \\
    &\frac{F^3_{2n-1}F^3_{2n}-F^3_{2n-1}-F^3_{2n}}{2}\\
    \ =\ &\frac{3}{4}(F^2_{2n-1}-F_{2n}^2) - \frac{3}{4} + \frac{3}{4}F_{2n-1}F_{2n}(F_{2n-1}^2-F^2_{2n}+F_{2n-1}F_{2n})+\\
    &\frac{(F^3_{2n-1}-1)(F^3_{2n}-1)}{2}\\
    \ =\ &\frac{3}{4}(F^2_{2n-1}-F_{2n}^2-1) + \frac{3}{4}F_{2n-1}F_{2n}(\underbrace{F_{2n-1}^2-F^2_{2n}+(F_{2n}-F_{2n-2})F_{2n}}_{=1})+\\
    &\frac{(F^3_{2n-1}-1)(F^3_{2n}-1)}{2}\\
    \ =\ &\frac{3}{4}(F^2_{2n-1}-F_{2n}^2+F_{2n-1}F_{2n}-1) + \frac{(F^3_{2n-1}-1)(F^3_{2n}-1)}{2}\\
    \ =\ &\frac{3}{4}(F^2_{2n-1}-F_{2n}F_{2n-2} - 1) + \frac{(F^3_{2n-1}-1)(F^3_{2n}-1)}{2}\\
    \ =\  &\frac{(F^3_{2n-1}-1)(F^3_{2n}-1)}{2}.
\end{align*}
\end{proof}

Of course, we can ask for the solution when $a$ and $b$ are different powers of consecutive Fibonacci numbers. 

\begin{prob}\normalfont
For $(i,j)\in \mathbb{N}^3$, find the solution $(x,y)$ when $(a,b) = (F_n^i, F_{n+1}^j)$ as $n$ varies.    
\end{prob}

\subsection{Balancing numbers and Lucas-balancing numbers}

A positive integer $n$ is called \textit{balancing} \cite{BP} if 
$$1 + 2 + \cdots + (n-1)\ =\ (n+1) + \cdots + (n+d), \mbox{ for some nonnegative integer }d.$$
The sequence of balancing numbers is denoted by $(B_n)_{n=1}^\infty$. By \cite[(9)]{BP}, $(B_n)_{n=1}^\infty$ can be defined recursively with $B_1 = 1, B_2 = 6$, and $B_{n} = 6B_{n-1} - B_{n-2}$ for $n\ge 3$ (\seqnum{A001109}). It is easy to verify that $\gcd(B_{n}, B_{n+1}) = \gcd(B_{2n-1}, B_{2n+1}) = 1$ for every $n\in \mathbb{N}$. 

For each $n$, the $n$\textsuperscript{th} \textit{Lucas-balancing} number $C_n$ is given by $C_n = \sqrt{8B_n^2+1}$. Results in \cite{Pan} suggest that $(C_n)_{n=1}^\infty$ is associated with $(B_n)_{n=1}^\infty$ in the same way Lucas numbers are associated with Fibonacci numbers. By \cite[Theorem 2.5]{Pan}, we can also define $(C_n)_{n=1}^\infty$ recursively as $C_1 = 3$, $C_2 = 17$, and $C_n = 6C_{n-1} - C_{n-2}$ for $n\ge 3$ (\seqnum{A001541}).

\begin{thm}\cite[Theorem 2.1]{Da}
    For $n\ge 1$,
    \begin{align}\label{e10} (x^{(0)}(B_{2n-1}, B_{2n}),\ y^{(0)}(B_{2n-1}, B_{2n}))&\ =\ \left(\frac{B_{2n-1}-1}{2}, b_{2n-1}\right);\\
    (x^{(1)}(B_{2n}, B_{2n+1}),\ y^{(1)}(B_{2n}, B_{2n+1})) &\ =\  \left(b_{2n+1}, \frac{B_{2n-1}-1}{2}\right),\nonumber
    \end{align}
where 
$$b_m \ =\ \frac{(1+\sqrt{2})^{2m-1} - (1-\sqrt{2})^{2m-1}}{4\sqrt{2}}-\frac{1}{2}, \quad m\in\mathbb{N}.$$
\end{thm}

\begin{proof}[Proof of \eqref{e10}]
The sequence $(b_n)_{n=1}^\infty$ can be defined recursively as: $b_1 = 0, b_2 = 2$, and $b_n = 6b_{n-1}-b_{n-2}+2$ for $n\ge 3$. Furthermore, by \cite[Corollary 3.4.2]{Ra},
\begin{equation}\label{e11}
    b_{n+1} - b_{n}\ =\ 2B_n, \mbox{ for all }n\in \mathbb{N}.
\end{equation}
We also use \cite[Theorem 5.1 (a)]{BP}, which states that
\begin{equation}\label{e12}B_{n+1}B_{n-1}-B^2_n\ =\ -1, \mbox{ for all }n\ge 2.\end{equation}

Observe that \eqref{e10} holds for $n = 1$. For $n\ge 2$, we have
    \begin{align*}
        &B_{2n-1}(B_{2n-1}-1)+2B_{2n}b_{2n-1}\\
        &\ =\ B_{2n-1}(B_{2n-1}-1) + B_{2n}(3b_{2n-1}-b_{2n-1})\\
        &\ =\ B^2_{2n-1}-B_{2n-1} + B_{2n}\left(\frac{b_{2n}+b_{2n-2}-2}{2} - b_{2n-1}\right)\\
        &\ =\ B^2_{2n-1} - B_{2n-1} + B_{2n}\left(\frac{b_{2n}-b_{2n-1}}{2}-\frac{b_{2n-1}-b_{2n-2}}{2}\right) - B_{2n}\\
        &\ =\ B^2_{2n-1} - B_{2n-1} + B_{2n}(B_{2n-1} - B_{2n-2}) - B_{2n}\quad\quad\mbox{ by \eqref{e11}}\\
        &\ =\ (B^2_{2n-1} - B_{2n-2}B_{2n}) +B_{2n-1}B_{2n} - B_{2n-1} - B_{2n}\\
        &\ =\ (B_{2n-1}-1)(B_{2n}-1)\quad\quad\mbox{ by \eqref{e12}}.
    \end{align*}
\end{proof}

The next theorem summarizes Davala's other neat identities that involve $B_n$'s and $C_n$'s \cite{Da}.

\begin{thm}\cite[Theorems 2.2, 2.3, 2.4, and 2.6]{Da}
    For $n\ge 1$, 
     \begin{align*}(x^{(0)}(B_{4n-3}, B_{4n-1}),\ y^{(0)}(B_{4n-3},B_{4n-1}))&\ =\ \left(\sum_{i=1}^{n-1}C_{4i}, \sum_{i=1}^{n-1}C_{4i}\right);\\
    (x^{(1)}(B_{4n-1}, B_{4n+1}),\ y^{(1)}(B_{4n-1}, B_{4n+1})) &\ =\  \left(\sum_{i=1}^n C_{4i}, \sum_{i=1}^{n-1} C_{4i}\right);\\
    \left(x^{(1)}\left(\frac{B_{4n}}{6}, \frac{B_{4n+2}}{6}\right),\ y^{(1)}\left(\frac{B_{4n}}{6}, \frac{B_{4n+2}}{6}\right)\right)&\ =\ \left(\sum_{i=1}^{2n}(-1)^i C_{2i}, \sum_{i=1}^{n-1}C_{4i}\right);\\
    \left(x^{(0)}\left(\frac{B_{4n-2}}{6}, \frac{B_{4n}}{6}\right),\ y^{(0)}\left(\frac{B_{4n-2}}{6}, \frac{B_{4n}}{6}\right)\right)&\ =\ \left(\sum_{i=1}^{n-1} C_{4i}, \sum_{i=1}^{2n-2}(-1)^iC_{2i}\right);\\
    \left(x^{(0)}(B_n, C_n),\ y^{(0)}(B_n, C_n)\right)&\ =\ (B_{n-1} + b_{n-1}, b_n);\\
    \left(x^{(1)}(C_{2n-1}, C_{2n}),\ y^{(1)}(C_{2n-1}, C_{2n})\right) &\ =\ \left(B_{2n}-\sum_{i=0}^{n-1}C_{2i}, \sum_{i=1}^{n-1}C_{2i}\right);\\
    \left(x^{(0)}(C_{2n}, C_{2n+1}),\ y^{(0)}(C_{2n}, C_{2n+1})\right) &\ =\ \left(\sum_{i=1}^n C_{2i}, B_{2n}-\sum_{i=0}^{n-1}C_{2i}\right).
    \end{align*}
\end{thm}

\section{Sequences having the Fibonacci recurrence and arbitrary initial terms} \label{Flikeseq}
The main goal of this section is to generalize Theorem \ref{ori} to consecutive terms of sequences that have the Fibonacci recurrence but take different initial values. For $(u, v)\in \mathbb{N}^2$ with $\gcd(u,v) = 1$, define the sequence $(t^{(u,v)}_{n})_{n=1}^\infty$ as follows: $t^{(u,v)}_1 = u$, $t^{(u,v)}_2 = v$, and $t^{(u,v)}_{n} = t^{(u,v)}_{n-1} + t^{(u,v)}_{n-2}$ for $n\ge 3$. It follows that
$$t^{(u, v)}_n\ =\ F_{n-2}u + F_{n-1}v, \mbox{ for all }n\ge 1.$$
We establish formulas for the unique nonnegative integral solution $(x,y)$ to either 
$$t^{(u,v)}_{n}x + t^{(u,v)}_{n+1}y \ =\ \frac{(t^{(u,v)}_{n}-1)(t^{(u,v)}_{n+1}-1)}{2}$$
or
$$1 + t^{(u,v)}_{n}x +  t^{(u,v)}_{n+1}y \ =\ \frac{( t^{(u,v)}_{n}-1)( t^{(u,v)}_{n+1}-1)}{2}.$$

Note that $F_n$ is even if and only if $3$ divides $n$. Since the solution $(x,y)$ depends not only the parity of $F_n$ but also on the parity of $n$, we need to consider six cases in total. In each case, the parity of $u$ and $v$ has further influences on the solution.

First, we record some preliminary results.

\begin{lem}\label{l1}
Given $(u,v)\in \mathbb{N}^2$ with $\gcd(u,v) = 1$ and odd $u\ge 3$, there is a unique odd integer $r\in [1,u-1]$ such that either $vr\equiv 1\mod u$ or $vr\equiv -1\mod u$. Similarly, there is a unique even integer $s\in [1, u-1]$ such that either $vs \equiv 1\mod u$ or $vs\equiv -1\mod u$. 
\end{lem}

\begin{proof}
Since $\gcd(u,v) = 1$, the set $\{1\cdot v, 2\cdot v, \ldots, u\cdot v\}$ is a complete modulo system of $u$. Hence, there are unique integers $x_1$ and $x_2\in [1, u-1]$ with $vx_1\equiv 1\mod u$ and $vx_2 \equiv -1\mod u$. It follows that $u$ divides $v(x_1+x_2)$, so $u$ divides $x_1 + x_2$. Furthermore, $2\le x_1 + x_2\le 2u-2$, so $x_1 + x_2 = u$. This implies that one of $x_1$ and $x_2$ is odd, while the other is even. 
\end{proof}

 Given an odd integer $u\ge 3$, we denote by 
$\mathbb{O}(u,v)$ the unique odd integer in $[1,u-1]$ such that $v\cdot \mathbb{O}(u,v)\equiv \pm 1\mod u$ 
and denote by
$\mathbb{E}(u,v)$ the unique even integer in $[1,u-1]$ such that $v\cdot \mathbb{E}(u,v)\equiv \pm 1\mod u$. Thanks to Lemma \ref{l1}, $\mathbb{O}(u,v)$ and $\mathbb{E}(u,v)$ are well-defined.

\begin{lem}\label{l2} 
Let $(u,v)\in \mathbb{N}^2$ with $\gcd(u,v) = 1$ and $2|u$. Pick an arbitrary odd $k\in [1, 2u-1]$. There is a unique odd integer $r\in [1, u]$ such that either $vr\equiv k\mod 2u$ or $vr\equiv -k\mod 2u$.
\end{lem}

\begin{proof}
We prove existence. Since $\gcd(2u, v) = 1$, the set $\{1\cdot v, 2\cdot v, \ldots, 2u\cdot v\}$ is a complete modulo system of $2u$. Let $x_1$ and $x_2$ be the unique integers in $[1, 2u-1]$ such that $vx_1\equiv k\mod 2u$ and $vx_2\equiv -k\mod 2u$. It follows that
$2u$ divides $(x_1 + x_2)$. Observe that $2\le x_1 + x_2 \le 4u-2$, so $x_1 + x_2 = 2u$. Hence, either $x_1\le u$ or $x_2\le u$. If $x_1\le u$, then because $2u$ divides $(vx_1 - k)$, $x_1$ must be odd and we set $r = x_1$. If $x_2\le u$, then because $2u$ divides $(vx_2+k)$, $x_2$ must be odd and we set $r = x_2$.

We prove uniqueness by contradiction. Suppose that there are two odd integers $x_1$ and $x_2$ in $[1, u]$ such that $vx_1\equiv k\mod 2u$ and $vx_2\equiv -k\mod 2u$. As above, $x_1 + x_2 = 2u$, so $x_1 = x_2 = u$, which is even. This contradicts the assumption that $x_1$ and $x_2$ are odd. 
\end{proof}

Given an even integer $u\ge 2$, we denote by $\mathbb{O}(u, v, k)$ the unique odd integer in $[1,u]$ such that 
$v\cdot \mathbb{O}(u, v, k)\equiv \pm k\mod 2u$.

\begin{thm}\label{genformodd}
Given $(u, v, n, r)\in \mathbb{Z}^4$ with odd $n$, it holds that

    \begin{align}\label{fe3}1&+\frac{1}{2}\left(rF_{n-1}+\frac{vr-1}{u}F_{n}-1\right)t^{(u,v)}_{n}+\nonumber\\
&\frac{1}{2}\left((u-r)F_{n-2}+\frac{(u-r)v+1}{u}F_{n-1}-1\right)t^{(u,v)}_{n+1}\ =\ \frac{(t^{(u,v)}_{n}-1)(t^{(u,v)}_{n+1}-1)}{2}
\end{align}
and
\begin{align}\label{fe4}&\frac{1}{2}\left(rF_{n-1}+\frac{vr+1}{u}F_{n}-1\right)t^{(u,v)}_{n} + \nonumber\\
&\frac{1}{2}\left((u-r)F_{n-2}+\frac{(u-r)v-1}{u}F_{n-1}-1\right)t^{(u,v)}_{n+1}\ =\ \frac{(t^{(u,v)}_{n}-1)(t^{(u,v)}_{n+1}-1)}{2}.
\end{align}
\end{thm}

\begin{proof}
We prove \eqref{fe3}. We have
\begin{align*}
    &2 + \left(rF_{n-1}+\frac{vr-1}{u}F_{n}-1\right)t^{(u,v)}_{n}+\\
    & \left((u-r)F_{n-2}+\frac{(u-r)v+1}{u}F_{n-1}-1\right)t^{(u,v)}_{n+1}\\
    \ =\ &1 + \left(rF_{n-1} + \frac{vr-1}{u}F_{n}\right)t^{(u,v)}_{n} + \\ 
    &\left((u-r)F_{n-2} + \frac{(u-r)v+1}{u}F_{n-1}\right)t^{(u,v)}_{n+1} - t^{(u,v)}_{n+1} - t^{(u, v)}_n +1\\
    \ =\ & 1 + urF_{n-2}F_{n-1} + (vr-1)F_{n-2}F_{n} + rvF^2_{n-1} + v \frac{vr-1}{u}F_{n-1}F_n + \\
    &u(u-r)F_{n-2}F_{n-1} + ((u-r)v+1)F^2_{n-1} + v(u-r)F_{n-2}F_n + \\
    &v\frac{(u-r)v+1}{u}F_{n-1}F_n - t^{(u,v)}_{n+1} - t^{(u, v)}_n +1\\
    \ =\ & 1 + u^2 F_{n-2}F_{n-1} + (uv-1)F_{n-2}F_n + (uv + 1)F^2_{n-1} + v^2 F_{n-1}F_n - \\
    &t^{(u,v)}_{n+1} - t^{(u, v)}_n + 1\\ 
    \ =\ & (1 + F^2_{n-1} - F_{n-2}F_n) + (uF_{n-2} + vF_{n-1})(uF_{n-1} + vF_n) - \\
    &t^{(u,v)}_{n+1} - t^{(u, v)}_n + 1\\
    \ =\ &t^{(u,v)}_{n}t^{(u,v)}_{n+1} - t^{(u,v)}_{n+1} - t^{(u, v)}_n + 1\\
    \ =\ &(t^{(u,v)}_n - 1)(t^{(u,v)}_{n+1}-1).
\end{align*}

To get \eqref{fe4}, we note
\begin{align*}
    &\left(rF_{n-1} + \frac{vr+1}{u}F_n - 1\right)t^{(u,v)}_n + \left((u-r)F_{n-2} + \frac{(u-r)v-1}{u}F_{n-1}-1\right)t_{n+1}^{(u,v)}\\
    \ =\ &  -1 + \left(rF_{n-1} + \frac{vr+1}{u}F_n\right)t^{(u,v)}_n + \left((u-r)F_{n-2} + \frac{(u-r)v-1}{u}F_{n-1}\right)t_{n+1}^{(u,v)}-\\
    &t^{(u,v)}_{n} - t^{(u,v)}_{n+1}+1\\
    \ =\ & -1 + urF_{n-2}F_{n-1} + (vr + 1)F_{n-2}F_n + vrF_{n-1}^2 + v\frac{vr+1}{u}F_{n-1}F_n + \\
    &(u-r)uF_{n-2}F_{n-1} + ((u-r)v-1)F^2_{n-1} + (u-r)vF_{n-2}F_n + \\
    & v \frac{(u-r)v-1}{u}F_{n-1}F_n -t^{(u,v)}_{n} - t^{(u,v)}_{n+1}+1 \\
    \ =\ &(-1 + F_{n-2}F_n  - F^2_{n-1}) + u^2F_{n-2}F_{n-1} + uvF_{n-2}F_n + uvF_{n-1}^2  + v^2 F_{n-1}F_n - \\
    &t^{(u,v)}_{n} - t^{(u,v)}_{n+1}+1 \\
    \ =\ &(uF_{n-2} + vF_{n-1})(uF_{n-1} + vF_n) - t^{(u,v)}_{n} - t^{(u,v)}_{n+1}+1\\
    \ =\ &t^{(u,v)}_nt^{(u,v)}_{n+1} - t^{(u,v)}_n - t^{(u,v)}_{n+1} + 1\ =\ (t^{(u,v)}_n - 1)(t^{(u,v)}_{n+1}-1).
\end{align*}
\end{proof}

Let 
\begin{align*}
&\Phi^{(0)}_{1}(u,v,n,r)\ :=\ \frac{1}{2}\left(rF_{n-1}+\frac{vr+1}{u}F_{n}-1\right),\\
&\Psi^{(0)}_{1}(u, v, n, r)\ :=\ \frac{1}{2}\left((u-r)F_{n-2}+\frac{(u-r)v-1}{u}F_{n-1}-1\right),\\
&\Phi^{(1)}_{1}(u,v,n,r)\ :=\ \frac{1}{2}\left(rF_{n-1}+\frac{vr-1}{u}F_{n}-1\right),\\
&\Psi^{(1)}_{1}(u, v, n, r)\ :=\ \frac{1}{2}\left((u-r)F_{n-2}+\frac{(u-r)v+1}{u}F_{n-1}-1\right).
\end{align*}
The subscript $1$ of $\Phi$ and $\Psi$ indicates that we are considering odd $n$. The superscript of $(0)$ or $(1)$ indicates whether ($\Phi$, $\Psi$) is a solution of \eqref{e1} or \eqref{e2}, respectively.

\begin{thm}\label{genformeven}
Given $(u, v, n, r)\in \mathbb{Z}^4$ with even $n$, it holds that

    \begin{align}\label{fe1}1&+\frac{1}{2}\left((u-r)F_{n-1}+\frac{(u-r)v+1}{u}F_{n}-1\right)t^{(u,v)}_{n}+\nonumber\\
&\frac{1}{2}\left(rF_{n-2}+\frac{vr-1}{u}F_{n-1}-1\right)t^{(u,v)}_{n+1}\ =\ \frac{(t^{(u,v)}_{n}-1)(t^{(u,v)}_{n+1}-1)}{2}
\end{align}
and
\begin{align}\label{fe2}&\frac{1}{2}\left((u-r)F_{n-1}+\frac{(u-r)v-1}{u}F_{n}-1\right)t^{(u,v)}_{n} + \nonumber\\
&\frac{1}{2}\left(rF_{n-2}+\frac{vr+1}{u}F_{n-1}-1\right)t^{(u,v)}_{n+1}\ =\ \frac{(t^{(u,v)}_{n}-1)(t^{(u,v)}_{n+1}-1)}{2}.
\end{align}
\end{thm}

\begin{proof}
To prove \eqref{fe1}, we see that
\begin{align*}
    &2+\left((u-r)F_{n-1}+\frac{(u-r)v+1}{u}F_{n}-1\right)t^{(u,v)}_{n} + \\
    &\left(rF_{n-2}+\frac{vr-1}{u}F_{n-1}-1\right)t^{(u,v)}_{n+1}\\
    \ =\ &1+\left((u-r)F_{n-1}+\frac{(u-r)v+1}{u}F_{n}\right)t^{(u,v)}_{n}+\\
    &\left(rF_{n-2}+\frac{vr-1}{u}F_{n-1}\right)t^{(u,v)}_{n+1}- t^{(u,v)}_{n}-t^{(u,v)}_{n+1}+1\\
    \ =\ &1+ u(u-r)F_{n-2}F_{n-1}+((u-r)v+1)F_{n-2}F_{n}+v(u-r)F_{n-1}^2 + \\
    &v\frac{(u-r)v+1}{u}F_{n-1}F_{n}+urF_{n-2}F_{n-1}+vrF_{n-2}F_{n} + \\
    &(vr-1)F_{n-1}^2+v\frac{vr-1}{u}F_{n-1}F_{n}- t^{(u,v)}_{n}-t^{(u,v)}_{n+1}+1\\
    \ =\ &1+ u^2F_{n-2}F_{n-1}+(uv+1)F_{n-2}F_{n}+(uv-1)F_{n-1}^2+v^2F_{n-1}F_{n}-\\
    &t^{(u,v)}_{n}-t^{(u,v)}_{n+1}+1\\
    \ =\ &(1 + F_{n-2}F_n - F_{n-1}^2) + u^2F_{n-2}F_{n-1}+uvF_{n-2}F_{n}+uvF_{n-1}^2+v^2F_{n-1}F_{n} - \\\
    &t^{(u,v)}_{n}-t^{(u,v)}_{n+1}+1\\
    \ =\ &(uF_{n-2} + vF_{n-1})(uF_{n-1}+vF_{n}) - t^{(u,v)}_{n}-t^{(u,v)}_{n+1}+1\\
    \ =\ &t^{(u,v)}_{n}t^{(u,v)}_{n+1}- t^{(u,v)}_{n}-t^{(u,v)}_{n+1}+1\\
    \ =\ &(t^{(u,v)}_{n}-1)(t^{(u,v)}_{n+1}-1).
\end{align*}
    Similarly, to obtain \eqref{fe2}, we have
\begin{align*}
     &\left((u-r)F_{n-1}+\frac{(u-r)v-1}{u}F_{n}-1\right)t^{(u,v)}_{n} + \left(rF_{n-2}+\frac{vr+1}{u}F_{n-1}-1\right)t^{(u,v)}_{n+1}\\
     \ =\ &-1+\left((u-r)F_{n-1}+\frac{(u-r)v-1}{u}F_{n}\right)t^{(u,v)}_{n} + \left(rF_{n-2}+\frac{vr+1}{u}F_{n-1}\right)t^{(u,v)}_{n+1} - \\
    &t^{(u,v)}_{n}-t^{(u,v)}_{n+1}+1\\
    \ =\ &-1+ u(u-r)F_{n-2}F_{n-1}+((u-r)v-1)F_{n-2}F_{n}+v(u-r)F_{n-1}^2 + \\
    &v\frac{(u-r)v-1}{u}F_{n-1}F_{n}+urF_{n-2}F_{n-1}+vrF_{n-2}F_{n} + \\
    &(vr+1)F_{n-1}^2+v\frac{vr+1}{u}F_{n-1}F_{n}- t^{(u,v)}_{n}-t^{(u,v)}_{n+1}+1\\
    \ =\ &-1+ u^2F_{n-2}F_{n-1}+(uv-1)F_{n-2}F_{n}+(uv+1)F_{n-1}^2+v^2F_{n-1}F_{n}-\\
    &t^{(u,v)}_{n}-t^{(u,v)}_{n+1}+1\\
    \ =\ &(-1 - F_{n-2}F_n + F_{n-1}^2) + u^2F_{n-2}F_{n-1}+uvF_{n-2}F_{n}+uvF_{n-1}^2+v^2F_{n-1}F_{n} - \\
    &t^{(u,v)}_{n}-t^{(u,v)}_{n+1}+1\\
    \ =\ &(uF_{n-2} + vF_{n-1})(uF_{n-1} + vF_{n}) - t^{(u,v)}_{n}-t^{(u,v)}_{n+1}+1\\
    \ =\ &t^{(u,v)}_{n}t^{(u,v)}_{n+1}- t^{(u,v)}_{n}-t^{(u,v)}_{n+1}+1\\
    \ =\ &(t^{(u,v)}_{n}-1)(t^{(u,v)}_{n+1}-1).
\end{align*}  
\end{proof}

Let
\begin{align*}
&\Phi^{(0)}_{2}(u,v,n,r)\ :=\ \frac{1}{2}\left((u-r)F_{n-1}+\frac{(u-r)v-1}{u}F_{n}-1\right),\\
&\Psi^{(0)}_{2}(u, v, n, r)\ :=\ \frac{1}{2}\left(rF_{n-2}+\frac{vr+1}{u}F_{n-1}-1\right),\\
&\Phi^{(1)}_{2}(u,v,n,r)\ :=\ \frac{1}{2}\left((u-r)F_{n-1}+\frac{(u-r)v+1}{u}F_{n}-1\right),\\
&\Psi^{(1)}_{2}(u, v, n, r)\ :=\ \frac{1}{2}\left(rF_{n-2}+\frac{vr-1}{u}F_{n-1}-1\right).
\end{align*}
The subscript $2$ of $\Phi$ and $\Psi$ indicates that we are considering even $n$. The superscript of $(0)$ or $(1)$ indicates whether ($\Phi$, $\Psi$) is a solution of \eqref{e1} or \eqref{e2}, respectively.

\subsection{When $\boldsymbol{n}$ is even}\label{neven}

\begin{thm}\label{0mod6}
Let $u,v\in \mathbb{N}$ with $\gcd(u,v)=1$ and let $n = 6k+6$ for some nonnegative integer $k$.
Set 
$r = \begin{cases}0, &\mbox{ if }u = 1;\\
\mathbb{E}(u,v), &\mbox{ if }u\mbox{ is odd and }u\ge 3 ;\\ \mathbb{O}(u,v,1), &\mbox{ if }u\mbox{ is even}.\end{cases}$

If $\begin{cases}u\mbox{ is odd}, u\ge 3 \mbox{ and }vr \equiv 1\mod u\mbox{, or }\\ u\mbox{ is even and }vr\equiv 1\mod 2u,\end{cases}$ then
\begin{equation}\label{ef7}
\begin{cases}
x^{(1)}(t^{(u,v)}_{n}, t^{(u,v)}_{n+1})&\ =\ \Phi^{(1)}_2(u,v, n, r),\\
y^{(1)}(t^{(u,v)}_{n}, t^{(u,v)}_{n+1})&\ =\ \Psi^{(1)}_2(u,v,n,r).
\end{cases}
\end{equation}

If $\begin{cases}u\mbox{ is odd and }vr \equiv -1\mod u\mbox{, or}\\ u\mbox{ is even and }vr\equiv -1\mod 2u,\end{cases}$ then
\begin{equation}\label{ef8}
\begin{cases}
x^{(0)}(t^{(u,v)}_{n}, t^{(u,v)}_{n+1})&\ =\ \Phi^{(0)}_2(u,v,n,r),\\
y^{(0)}(t^{(u,v)}_{n}, t^{(u,v)}_{n+1})&\ =\ \Psi^{(0)}_2(u,v,n,r).
\end{cases}
\end{equation}
\end{thm}

\begin{proof}
Thanks to Theorem \ref{genformeven}, we need only to prove that \eqref{ef7} and \eqref{ef8} give nonnegative integers.

Case 1: $\begin{cases}u\mbox{ is odd}, u\ge 3,\mbox{ and }vr \equiv 1\mod u\mbox{, or}\\ u\mbox{ is even and }vr\equiv 1\mod 2u.\end{cases}$

\begin{table}[H]
\centering
\begin{tabular}{ | c| c| c| c| c| c|}
\hline
$r$& $u-r$ & $\frac{(u-r)v+1}{u}$ & $\frac{vr-1}{u}$ & $\Phi^{(1)}_2\ge 0?$ & $\Psi^{(1)}_2\ge 0?$ \\
\hline
$\ge 1$ & $\ge 0$ & $\ge 1$ & $\ge 0$ & \checkmark & \checkmark   \\
\hline
\end{tabular}
\caption{Case 1's nonnegative solutions for $n = 6k+6$.}
\end{table}

    \begin{table}[H]
\centering
\begin{tabular}{ |c| c| c| c| c| c| c|}
\hline
$n = 6k+6, k\ge 0, u\ge 2$ &$r$& $u-r$ & $\frac{(u-r)v+1}{u}$ & $\frac{vr-1}{u}$ & $\Phi^{(1)}_2\in\mathbb{Z}?$ & $\Psi^{(1)}_2\in \mathbb{Z}?$ \\
\hline
$2\nmid u$, $u|(vr-1)$& even & odd & n/a & odd & \checkmark & \checkmark   \\
\hline
$2|u$, $(2u)|(vr-1)$& odd & odd & odd & even & \checkmark & \checkmark\\
\hline
\end{tabular}
\caption{Case 1's integral solutions for $n = 6k+6$.}
\end{table}

Case 2: $\begin{cases}u\mbox{ is odd and } vr \equiv -1\mod u, \mbox{ or}\\ u\mbox{ is even and }vr\equiv -1\mod 2u.\end{cases}$ Observe that $r < u$ because $r = u$ implies that $u = 1$, in which case $r = 0 < u$, a contradiction. 

\begin{table}[H]
\centering
\begin{tabular}{ | c| c| c| c| c| c|}
\hline
$r$& $u-r$ & $\frac{(u-r)v-1}{u}$ & $\frac{vr+1}{u}$ & $\Phi^{(0)}_2\ge 0?$ & $\Psi^{(0)}_2\ge 0?$ \\
\hline
$\ge 0$ & $\ge 1$ & $\ge 0$ & $\ge 1$ & \checkmark & \checkmark   \\
\hline
\end{tabular}
\caption{Case 2's  nonnegative solutions for $n = 6k+6$.}
\end{table}

    \begin{table}[H]
\centering
\begin{tabular}{ |c| c| c| c| c| c| c|}
\hline
$n = 6k+6, k\ge 0$ &$r$& $u-r$ & $\frac{(u-r)v-1}{u}$ & $\frac{vr+1}{u}$ & $\Phi^{(0)}_2\in\mathbb{Z}?$ & $\Psi^{(0)}_2\in \mathbb{Z}?$ \\
\hline
$2\nmid u$, $u|(vr+1)$& even & odd & n/a & odd & \checkmark & \checkmark   \\
\hline
$2|u$, $(2u)|(vr+1)$& odd & odd & odd & even & \checkmark & \checkmark\\
\hline
\end{tabular}
\caption{Case 2's integral solutions for $n = 6k+6$.}
\end{table}
\end{proof}

\begin{thm}\label{6cases2mod6}
Let $u,v\in \mathbb{N}$ with $\gcd(u,v)=1$ and let $n = 6k+2$ for some nonnegative integer $k$.
Set  $r = \begin{cases}0, &\mbox{ if }u = 1\mbox{ and }v\mbox{ is odd};\\ 1, &\mbox{ if }u = 1\mbox{ and }v\mbox{ is even};\\ 
\mathbb{E}(u,v), &\mbox{ if }u\mbox{ is odd}, u\ge 3, \mbox{ and }v\mbox{ is odd};\\
\mathbb{O}(u,v), &\mbox{ if }u\mbox{ is odd}, u\ge 3, \mbox{ and }v\mbox{ is even};\\
\mathbb{O}(u,v,u+1), &\mbox{ if }u\mbox{ is even}.\end{cases}$

If $\begin{cases}u = 1\mbox{ and }v\mbox{ is even, or}\\u\mbox{ is odd}, u\ge 3,\mbox{ and }vr \equiv 1\mod u\mbox{, or }\\ u\mbox{ is even and }vr\equiv (u+1)\mod 2u,\end{cases}$ then
\begin{equation}\label{ef5}
\begin{cases}
x^{(1)}(t^{(u,v)}_{n}, t^{(u,v)}_{n+1})&\ =\ \Phi^{(1)}_2(u, v, n, r),\\
y^{(1)}(t^{(u,v)}_{n}, t^{(u,v)}_{n+1})&\ =\ \Psi^{(1)}_2(u, v, n, r).
\end{cases}
\end{equation}

If $\begin{cases} u = 1\mbox{ and }v\mbox{ is odd};\\
u\mbox{ is odd}, u\ge 3, \mbox{  and }vr \equiv -1\mod u\mbox{, or }\\ u\mbox{ is even and }vr\equiv -(u+1)\mod 2u,\end{cases}$ then
\begin{equation}
\begin{cases}\label{ef6}
x^{(0)}(t^{(u,v)}_{n}, t^{(u,v)}_{n+1})&\ =\ \Phi^{(0)}_2(u, v, n, r),\\
y^{(0}(t^{(u,v)}_{n}, t^{(u,v)}_{n+1})&\ =\ \Psi^{(0)}_2(u, v, n, r).
\end{cases}
\end{equation}
\end{thm}

\begin{proof}
    Thanks to Theorem \ref{genformeven}, we need only to verify that \eqref{ef5} and \eqref{ef6} give nonnegative integral solutions.

    Case 1: $u = 1$ and $v$ is even. Then $r = 1$, so
    \begin{align*}
    &\Phi^{(1)}_2(1, v, n, 1)\ =\ \frac{1}{2}(F_n - 1)\ \ge\ 0,\mbox{ and }\\
    &\Psi^{(1)}_2(1, v, n, 1)\ =\ \frac{1}{2}(F_{n-2} + (v-1)F_{n-1}-1)\ \ge\ 0.
    \end{align*}
    Both $(F_n-1)/2$ and $(F_{n-2}+(v-1)F_{n-1}-1)/2$ are integers because 
    $F_n$ and $F_{n-1}$ are odd, and $F_{n-2}$ is even.

    Case 2: $u = 1$ and $v$ is odd. Then $r = 0$, so
    \begin{align*}
    &\Phi^{(0)}_2(1, v, n, 0)\ =\ \frac{1}{2}(F_{n-1} + (v-1)F_n- 1)\ \ge\ 0,\mbox{ and }\\
    &\Psi^{(0)}_2(1, v, n, 0)\ =\ \frac{1}{2}(F_{n-1}-1)\ \ge\ 0.
    \end{align*}
    Both $(F_{n-1}+(v-1)F_n - 1)/2$ and $(F_{n-1}-1)/2$ are integers because $F_{n-1}$, $F_n$, and $v$ are odd.
    
    Case 3: $\begin{cases}u\mbox{ is odd}, u\ge 3,\mbox{ and }vr \equiv 1\mod u\mbox{, or }\\ u\mbox{ is even and }vr\equiv (u+1)\mod 2u.\end{cases}$ It follows from our choice of $r$ that for odd $u$, $vr\ge 2$, and for even $u$, $(vr-1)/u$ is nonnegative and odd. Hence, $(vr-1)/u\ge 1$. 

\begin{table}[H]
\centering
\begin{tabular}{ | c| c| c| c| c| c|}
\hline
$r$& $u-r$ & $\frac{(u-r)v+1}{u}$ & $\frac{vr-1}{u}$ & $\Phi^{(1)}_2\ge 0?$ & $\Psi^{(1)}_2\ge 0?$ \\
\hline
 $\ge 1$ & $\ge 0$ & $\ge 1$ & $\ge 1$ & \checkmark & \checkmark   \\
\hline
\end{tabular}
\caption{Case 3's nonnegative solutions for $n = 6k+2$.}
\end{table}

    \begin{table}[H]
\centering
\begin{tabular}{ |c| c| c| c| c| c| c|}
\hline
$n = 6k+2, k\ge 0, u\ge 2$ &$r$& $u-r$ & $\frac{(u-r)v+1}{u}$ & $\frac{vr-1}{u}$ & $\Phi^{(1)}_2\in\mathbb{Z}?$ & $\Psi^{(1)}_2\in \mathbb{Z}?$ \\
\hline
$2\nmid u$, $u|(vr-1)$, $2|v$ & odd & even & odd & odd & \checkmark & \checkmark   \\
\hline
$2\nmid u$,  $u|(vr-1)$, $2\nmid v$ & even & odd & even & odd & \checkmark & \checkmark   \\
\hline
$2|u$, $(2u)|(vr-u-1)$& odd & odd & even & odd & \checkmark & \checkmark\\
\hline
\end{tabular}
\caption{Case 3's integral solutions for $n = 6k+2$.}
\end{table}

    Case 4: $\begin{cases}u\mbox{ is odd}, u\ge 3, \mbox{ and } vr \equiv -1\mod u\mbox{, or }\\ u\mbox{ is even and }vr\equiv -(u+1)\mod 2u.\end{cases}$ 
    Observe that $u > r$ because $u = r$ implies that $u = 1$.

    \begin{table}[H]
\centering
\begin{tabular}{  |c| c| c| c| c| c|}
\hline
$r$& $u-r$ & $\frac{(u-r)v-1}{u}$ & $\frac{vr+1}{u}$ & $\Phi^{(0)}_2\ge 0?$ & $\Psi^{(0)}_2\ge 0?$ \\
\hline
$\ge 1$ & $\ge 1$ & $\ge 0$ & $\ge 1$ & \checkmark & \checkmark   \\
\hline
\end{tabular}
\caption{Case 4's nonnegative solutions for $n = 6k+2$.}
\end{table}

    \begin{table}[H]
\centering
\begin{tabular}{ |c| c| c| c| c| c| c|}
\hline
$n = 6k+2, k\ge 0, u\ge 2$ &$r$& $u-r$ & $\frac{(u-r)v-1}{u}$ & $\frac{vr+1}{u}$ & $\Phi^{(0)}_2\in\mathbb{Z}?$ & $\Psi^{(0)}_2\in \mathbb{Z}?$ \\
\hline
$2\nmid u$, $u|(vr+1)$, $2|v$ & odd & even & odd & odd & \checkmark & \checkmark   \\
\hline
$2\nmid u$, $u|(vr+1)$, $2\nmid v$ & even & odd & even & odd & \checkmark & \checkmark   \\
\hline
$2|u$, $(2u)|(vr+u+1)$& odd & odd & even & odd & \checkmark & \checkmark\\
\hline
\end{tabular}
\caption{Case 4's integral solutions for $n = 6k+2$.}
\end{table}
\end{proof}

\begin{thm}\label{4mod6}
Let $u,v\in \mathbb{N}$ with $\gcd(u,v)=1$ and let $n = 6k+4$ for some nonnegative integer $k$.
Set 
$r = \begin{cases}1, &\mbox{ if }u = 1;\\
\mathbb{O}(u,v), &\mbox{ if }u\mbox{ is odd and }u\ge 3 ;\\ \mathbb{O}(u,v,1), &\mbox{ if }u\mbox{ is even}.\end{cases}$

If $\begin{cases}u\mbox{ is odd and }vr \equiv 1\mod u\mbox{, or }\\ u\mbox{ is even and }vr\equiv 1\mod 2u,\end{cases}$ then
\begin{equation}\label{ef9}
\begin{cases}
x^{(1)}(t^{(u,v)}_{n}, t^{(u,v)}_{n+1})&\ =\ \Phi^{(1)}_2(u,v, n, r),\\
y^{(1)}(t^{(u,v)}_{n}, t^{(u,v)}_{n+1})&\ =\ \Psi^{(1)}_2(u,v,n,r).
\end{cases}
\end{equation}

If $\begin{cases}u\mbox{ is odd}, u\ge 3, \mbox{ and }vr \equiv -1\mod u\mbox{, or}\\ u\mbox{ is even and }vr\equiv -1\mod 2u,\end{cases}$ then
\begin{equation}\label{ef10}
\begin{cases}
x^{(0)}(t^{(u,v)}_{n}, t^{(u,v)}_{n+1})&\ =\ \Phi^{(0)}_2(u,v,n,r),\\
y^{(0)}(t^{(u,v)}_{n}, t^{(u,v)}_{n+1})&\ =\ \Psi^{(0)}_2(u,v,n,r).
\end{cases}
\end{equation}
\end{thm}

\begin{proof}
Thanks to Theorem \ref{genformeven}, we need only to prove that \eqref{ef9} and \eqref{ef10} give nonnegative integers.

Case 1: $\begin{cases}u\mbox{ is odd and }vr \equiv 1\mod u\mbox{, or}\\ u\mbox{ is even and }vr\equiv 1\mod 2u.\end{cases}$ 

\begin{table}[H]
\centering
\begin{tabular}{ | c| c| c| c| c| c|}
\hline
$r$& $u-r$ & $\frac{(u-r)v+1}{u}$ & $\frac{vr-1}{u}$ & $\Phi^{(1)}_2\ge 0?$ & $\Psi^{(1)}_2\ge 0?$ \\
\hline
$\ge 1$ & $\ge 0$ & $\ge 1$ & $\ge 0$ & \checkmark & \checkmark   \\
\hline
\end{tabular}
\caption{Case 1's nonnegative solutions for $n = 6k+4$.}
\end{table}

    \begin{table}[H]
\centering
\begin{tabular}{ |c| c| c| c| c| c| c|}
\hline
$n = 6k+4, k\ge 0$ &$r$& $u-r$ & $\frac{(u-r)v+1}{u}$ & $\frac{vr-1}{u}$ & $\Phi^{(1)}_2\in\mathbb{Z}?$ & $\Psi^{(1)}_2\in \mathbb{Z}?$ \\
\hline
$2\nmid u$, $u|(vr-1)$& odd & even & odd & n/a & \checkmark & \checkmark   \\
\hline
$2|u$, $(2u)|(vr-1)$& odd & odd & odd & even & \checkmark & \checkmark\\
\hline
\end{tabular}
\caption{Case 1's integral solutions for $n = 6k+4$.}
\end{table}

Case 2: $\begin{cases}u\mbox{ is odd}, u\ge 3,\mbox{ and } vr \equiv -1\mod u, \mbox{ or}\\ u\mbox{ is even and }vr\equiv -1\mod 2u.\end{cases}$ Observe that $r < u$ because $r = u$ implies that $u = 1$.

\begin{table}[H]
\centering
\begin{tabular}{ | c| c| c| c| c| c|}
\hline
$r$& $u-r$ & $\frac{(u-r)v-1}{u}$ & $\frac{vr+1}{u}$ & $\Phi^{(0)}_2\ge 0?$ & $\Psi^{(0)}_2\ge 0?$ \\
\hline
$\ge 1$ & $\ge 1$ & $\ge 0$ & $\ge 1$ & \checkmark & \checkmark   \\
\hline
\end{tabular}
\caption{Case 2's nonnegative solutions for $n = 6k+4$.}
\end{table}

    \begin{table}[H]
\centering
\begin{tabular}{ |c| c| c| c| c| c| c|}
\hline
$n = 6k+4, k\ge 0, u\ge 2$ &$r$& $u-r$ & $\frac{(u-r)v-1}{u}$ & $\frac{vr+1}{u}$ & $\Phi^{(0)}_2\in\mathbb{Z}?$ & $\Psi^{(0)}_2\in \mathbb{Z}?$ \\
\hline
$2\nmid u$, $u|(vr+1)$& odd & even & odd & n/a & \checkmark & \checkmark   \\
\hline
$2|u$, $(2u)|(vr+1)$& odd & odd & odd & even & \checkmark & \checkmark\\
\hline
\end{tabular}
\caption{Case 2's integral solutions for $n = 6k+4$.}
\end{table}
\end{proof}

\subsection{When $\boldsymbol{n}$ is odd}\label{nodd}

\begin{thm}\label{1mod6}
Let $u,v\in \mathbb{N}$ with $\gcd(u,v)=1$ and let $n = 6k+1$ for some nonnegative integer $k$.
Set 
$r = \begin{cases}0, &\mbox{ if }u = 1;\\
\mathbb{E}(u,v), &\mbox{ if }u\mbox{ is odd and }u\ge 3 ;\\ \mathbb{O}(u,v,u+1), &\mbox{ if }u\mbox{ is even}.\end{cases}$

If $v = n = 1$, then $x^{(0)}(t^{(u,1)}_1, t^{(u,1)}_2) = y^{(0)}(t^{(u,1)}_1, t^{(u, 1)}_2) = 0$.

If $(v,n)\neq (1,1)$ and $\begin{cases}u\mbox{ is odd}, u\ge 3,\mbox{ and }vr \equiv 1\mod u\mbox{, or }\\ u\mbox{ is even and }vr\equiv u+1\mod 2u,\end{cases}$ then
\begin{equation}\label{ef11}
\begin{cases}
x^{(1)}(t^{(u,v)}_{n}, t^{(u,v)}_{n+1})&\ =\ \Phi^{(1)}_1(u,v, n, r),\\
y^{(1)}(t^{(u,v)}_{n}, t^{(u,v)}_{n+1})&\ =\ \Psi^{(1)}_1(u,v,n,r).
\end{cases}
\end{equation}

If $ (v,n)\neq (1,1)$ and $\begin{cases}u\mbox{ is odd and }vr \equiv -1\mod u\mbox{, or}\\ u\mbox{ is even and }vr\equiv -u-1\mod 2u,\end{cases}$ then
\begin{equation}\label{ef12}
\begin{cases}
x^{(0)}(t^{(u,v)}_{n}, t^{(u,v)}_{n+1})&\ =\ \Phi^{(0)}_1(u,v,n,r),\\
y^{(0)}(t^{(u,v)}_{n}, t^{(u,v)}_{n+1})&\ =\ \Psi^{(0)}_1(u,v,n,r).
\end{cases}
\end{equation}
\end{thm}

\begin{proof}
Thanks to Theorem \ref{genformodd}, we need only to prove that \eqref{ef11} and \eqref{ef12} give nonnegative integers.

Case 1: $v = n = 1$. We have $t^{(u, 1)}_{1} = u$ and $t^{(u,1)}_2 = 1$, so
$$t^{(u, 1)}_{1}\cdot 0 + t^{(u,1)}_2\cdot 0\ =\ \frac{(t^{(u, 1)}_{1}-1)(t^{(u, 1)}_{2}-1)}{2}.$$

Case 2: $(v, n)\neq (1,1)$ and $\begin{cases}u\mbox{ is odd}, u\ge 3, \mbox{ and }vr \equiv 1\mod u\mbox{, or}\\ u\mbox{ is even and }vr\equiv u+1\mod 2u.\end{cases}$ 
Observe that $u = r$ implies that $u = 1$, which does not belong to the case we are considering. Hence, $u-r\ge 1$ and thus, $\Psi^{(1)}_1(u,v,n,r)\ge 0$. 

Furthermore, if $n = 1$, then $v\ge 2$; if $n > 1$, then $n\ge 7$. These combined with $r\ge 1$ give $\Phi^{(1)}_1(u, v, n, r)\ge 0$.

    \begin{table}[H]
\centering
\begin{tabular}{ |c| c| c| c| c| c| c|}
\hline
$n = 6k+1, k\ge 0, u\ge 2$ &$r$& $u-r$ & $\frac{(u-r)v+1}{u}$ & $\frac{vr-1}{u}$ & $\Phi^{(1)}_1\in\mathbb{Z}?$ & $\Psi^{(1)}_1\in \mathbb{Z}?$ \\
\hline
$2\nmid u$, $u|(vr-1)$& even & odd & n/a & odd & \checkmark & \checkmark   \\
\hline
$2|u$, $(2u)|(vr-u-1)$& odd & odd & even & odd & \checkmark & \checkmark\\
\hline
\end{tabular}
\caption{Case 1's integral solutions for $n = 6k+1$.}
\end{table}

Case 3: $(v, n)\neq (1,1)$ and $\begin{cases}u\mbox{ is odd and } vr \equiv -1\mod u, \mbox{ or}\\ u\mbox{ is even and }vr\equiv -(u+1)\mod 2u.\end{cases}$ Observe that $r < u$ because $r = u$ implies that $r = u = 1$, contradicting our choice of $r$.

\begin{table}[H]
\centering
\begin{tabular}{ | c| c| c| c| c| c|}
\hline
$r$& $u-r$ & $\frac{(u-r)v-1}{u}$ & $\frac{vr+1}{u}$ & $\Phi^{(0)}_1\ge 0?$ & $\Psi^{(0)}_1\ge 0?$ \\
\hline
$\ge 0$ & $\ge 1$ & $\ge 0$ & $\ge 1$ & \checkmark & \checkmark   \\
\hline
\end{tabular}
\caption{Case 3's nonnegative solutions for $n = 6k+1$.}
\end{table}

    \begin{table}[H]
\centering
\begin{tabular}{ |c| c| c| c| c| c| c|}
\hline
$n = 6k+1, k\ge 0$ &$r$& $u-r$ & $\frac{(u-r)v-1}{u}$ & $\frac{vr+1}{u}$ & $\Phi^{(0)}_1\in\mathbb{Z}?$ & $\Psi^{(0)}_1\in \mathbb{Z}?$ \\
\hline
$2\nmid u$, $u|(vr+1)$& even & odd & n/a & odd & \checkmark & \checkmark   \\
\hline
$2|u$, $(2u)|(vr+u+1)$& odd & odd & even & odd & \checkmark & \checkmark\\
\hline
\end{tabular}
\caption{Case 3's integral solutions for $n = 6k+1$.}
\end{table}
\end{proof}

\begin{thm}\label{3mod6}
Let $u,v\in \mathbb{N}$ with $\gcd(u,v)=1$ and let $n = 6k+3$ for some nonnegative integer $k$.
Set 
$r = \begin{cases}1, &\mbox{ if }u = 1;\\
\mathbb{O}(u,v), &\mbox{ if }u\mbox{ is odd and }u\ge 3 ;\\ \mathbb{O}(u,v,u+1), &\mbox{ if }u\mbox{ is even}.\end{cases}$

If $\begin{cases}u\mbox{ is odd and }vr \equiv 1\mod u\mbox{, or }\\ u\mbox{ is even and }vr\equiv u+1\mod 2u,\end{cases}$ then
\begin{equation}\label{ef13}
\begin{cases}
x^{(1)}(t^{(u,v)}_{n}, t^{(u,v)}_{n+1})&\ =\ \Phi^{(1)}_1(u,v, n, r),\\
y^{(1)}(t^{(u,v)}_{n}, t^{(u,v)}_{n+1})&\ =\ \Psi^{(1)}_1(u,v,n,r).
\end{cases}
\end{equation}

If $\begin{cases}u\mbox{ is odd}, u\ge 3,\mbox{ and }vr \equiv -1\mod u\mbox{, or}\\ u\mbox{ is even and }vr\equiv -(u+1)\mod 2u,\end{cases}$ then
\begin{equation}\label{ef14}
\begin{cases}
x^{(0)}(t^{(u,v)}_{n}, t^{(u,v)}_{n+1})&\ =\ \Phi^{(0)}_1(u,v,n,r),\\
y^{(0)}(t^{(u,v)}_{n}, t^{(u,v)}_{n+1})&\ =\ \Psi^{(0)}_1(u,v,n,r).
\end{cases}
\end{equation}
\end{thm}

\begin{proof}
Thanks to Theorem \ref{genformodd}, we need only to prove that \eqref{ef13} and \eqref{ef14} give nonnegative integers.

Case 1: $\begin{cases}u\mbox{ is odd and }vr \equiv 1\mod u\mbox{, or}\\ u\mbox{ is even and }vr\equiv u+1\mod 2u.\end{cases}$

\begin{table}[H]
\centering
\begin{tabular}{ | c| c| c| c| c| c|}
\hline
$r$& $u-r$ & $\frac{(u-r)v+1}{u}$ & $\frac{vr-1}{u}$ & $\Phi^{(1)}_1\ge 0?$ & $\Psi^{(1)}_1\ge 0?$ \\
\hline
$\ge 1$ & $\ge 0$ & $\ge 1$ & $\ge 0$ & \checkmark & \checkmark   \\
\hline
\end{tabular}
\caption{Case 1's nonnegative solutions for $n = 6k+3$.}
\end{table}

    \begin{table}[H]
\centering
\begin{tabular}{ |c| c| c| c| c| c| c|}
\hline
$n = 6k+3, k\ge 0$ &$r$& $u-r$ & $\frac{(u-r)v+1}{u}$ & $\frac{vr-1}{u}$ & $\Phi^{(1)}_1\in\mathbb{Z}?$ & $\Psi^{(1)}_1\in \mathbb{Z}?$ \\
\hline
$2\nmid u$, $u|(vr-1)$& odd & even & odd & n/a & \checkmark & \checkmark   \\
\hline
$2|u$, $(2u)|(vr-u-1)$& odd & odd & even & odd & \checkmark & \checkmark\\
\hline
\end{tabular}
\caption{Case 1's integral solutions for $n = 6k+3$.}
\end{table}

Case 2: $\begin{cases}u\mbox{ is odd}, u\ge 3,\mbox{ and } vr \equiv -1\mod u, \mbox{ or}\\ u\mbox{ is even and }vr\equiv -(u+1)\mod 2u.\end{cases}$ Observe that $r < u$ because $r = u$ implies that $u = 1$.

\begin{table}[H]
\centering
\begin{tabular}{ | c| c| c| c| c| c|}
\hline
 $r$& $u-r$ & $\frac{(u-r)v-1}{u}$ & $\frac{vr+1}{u}$ & $\Phi^{(0)}_1\ge 0?$ & $\Psi^{(0)}_1\ge 0?$ \\
\hline
  $\ge 1$ & $\ge 1$ & $\ge 0$ & $\ge 1$ & \checkmark & \checkmark   \\
\hline
\end{tabular}
\caption{Case 2's nonnegative solutions for $n = 6k+3$.}
\end{table}

    \begin{table}[H]
\centering
\begin{tabular}{ |c| c| c| c| c| c| c|}
\hline
$n = 6k+3, k\ge 0, u\ge 2$ &$r$& $u-r$ & $\frac{(u-r)v-1}{u}$ & $\frac{vr+1}{u}$ & $\Phi^{(0)}_1\in\mathbb{Z}?$ & $\Psi^{(0)}_1\in \mathbb{Z}?$ \\
\hline
$2\nmid u$, $u|(vr+1)$& odd & even & odd & n/a & \checkmark & \checkmark   \\
\hline
$2|u$, $(2u)|(vr+u+1)$& odd & odd & even & odd & \checkmark & \checkmark\\
\hline
\end{tabular}
\caption{Case 2's integral solutions for $n = 6k+3$.}
\end{table}
\end{proof}

\begin{thm}\label{6cases5mod6}
Let $u,v\in \mathbb{N}$ with $\gcd(u,v)=1$ and let $n = 6k+5$ for some nonnegative integer $k$.
Set  $r = \begin{cases}1, &\mbox{ if }u = 1\mbox{ and }v\mbox{ is odd};\\ 0, &\mbox{ if }u = 1\mbox{ and }v\mbox{ is even};\\ \mathbb{E}(u,v), &\mbox{ if }u\mbox{ is odd}, u\ge 3, \mbox{ and }v\mbox{ is even};\\
\mathbb{O}(u,v), &\mbox{ if }u\mbox{ is odd}, u\ge 3, \mbox{ and }v\mbox{ is odd};\\
\mathbb{O}(u,v,1), &\mbox{ if }u\mbox{ is even}.\end{cases}$

If $\begin{cases}u = 1\mbox{ and }v\mbox{ is odd, or}\\u\mbox{ is odd}, u\ge 3,\mbox{ and }vr \equiv 1\mod u\mbox{, or }\\ u\mbox{ is even and }vr\equiv 1\mod 2u,\end{cases}$ then
\begin{equation}\label{ef15}
\begin{cases}
x^{(1)}(t^{(u,v)}_{n}, t^{(u,v)}_{n+1})&\ =\ \Phi^{(1)}_1(u, v, n, r),\\
y^{(1)}(t^{(u,v)}_{n}, t^{(u,v)}_{n+1})&\ =\ \Psi^{(1)}_1(u, v, n, r).
\end{cases}
\end{equation}

If $\begin{cases} u = 1\mbox{ and }v\mbox{ is even, or}\\
u\mbox{ is odd}, u\ge 3, \mbox{  and }vr \equiv -1\mod u\mbox{, or }\\ u\mbox{ is even and }vr\equiv -1\mod 2u,\end{cases}$ then
\begin{equation}
\begin{cases}\label{ef16}
x^{(0)}(t^{(u,v)}_{n}, t^{(u,v)}_{n+1})&\ =\ \Phi^{(0)}_1(u, v, n, r),\\
y^{(0}(t^{(u,v)}_{n}, t^{(u,v)}_{n+1})&\ =\ \Psi^{(0)}_1(u, v, n, r).
\end{cases}
\end{equation}
\end{thm}

\begin{proof}
    Thanks to Theorem \ref{genformodd}, we need only to verify that \eqref{ef15} and \eqref{ef16} give nonnegative integral solutions.

    Case 1: $u = 1$ and $v$ is odd. Then $r = 1$. We have
    \begin{align*}
    &\Phi^{(1)}_{1}(1,v,n,1)\ =\ \frac{1}{2}(F_{n-1} + (v-1)F_n- 1)\ \ge\ 0,\mbox{ and }\\
    &\Psi^{(1)}_{1}(1, v, n, 1)\ =\ \frac{1}{2}(F_{n-1} -1)\ \ge\ 0.
    \end{align*}
    Both $\Phi^{(1)}_{1}(1,v,n,1)$ and $\Psi^{(1)}_{1}(1, v, n, 1)$ are integers because 
    $F_n$, $F_{n-1}$, and $v$ are odd.

    Case 2: $u = 1$ and $v$ is even. Then $r = 0$. We have
    \begin{align*}
    &\Phi^{(0)}_{1}(1,v,n,0)\ =\ \frac{1}{2}(F_n- 1)\ \ge\ 0,\mbox{ and }\\
    &\Psi^{(0)}_{1}(1, v, n, 0)\ =\ \frac{1}{2}(F_{n-2} + (v-1)F_{n-1}-1)\ \ge\ 0.
    \end{align*}
    Both $\Phi^{(0)}_{1}(1,v,n,0)$ and $\Psi^{(0)}_{1}(1, v, n, 0)$ are integers because $F_{n-1}$, $F_n$, and $v-1$ are odd, while $F_{n-2}$ is even.
    
    Case 3: $\begin{cases}u\mbox{ is odd}, u\ge 3,\mbox{ and }vr \equiv 1\mod u\mbox{, or }\\ u\mbox{ is even and }vr\equiv 1\mod 2u.\end{cases}$

\begin{table}[H]
\centering
\begin{tabular}{ | c| c| c| c| c| c|}
\hline
$r$& $u-r$ & $\frac{(u-r)v+1}{u}$ & $\frac{vr-1}{u}$ & $\Phi^{(1)}_1\ge 0?$ & $\Psi^{(1)}_1\ge 0?$ \\
\hline
 $\ge 1$ & $\ge 0$ & $\ge 1$ & $\ge 0$ & \checkmark & \checkmark   \\
\hline
\end{tabular}
\caption{Case 3's nonnegative solutions for $n = 6k+5$.}
\end{table}

    \begin{table}[H]
\centering
\begin{tabular}{ |c| c| c| c| c| c| c|}
\hline
$n = 6k+5, k\ge 0, u\ge 2$ &$r$& $u-r$ & $\frac{(u-r)v+1}{u}$ & $\frac{vr-1}{u}$ & $\Phi^{(1)}_1\in\mathbb{Z}?$ & $\Psi^{(1)}_1\in \mathbb{Z}?$ \\
\hline
$2\nmid u$, $u|(vr-1)$, $2|v$ & even & odd & odd & odd & \checkmark & \checkmark   \\
\hline
$2\nmid u$,  $u|(vr-1)$, $2\nmid v$ & odd & even & odd & even & \checkmark & \checkmark   \\
\hline
$2|u$, $(2u)|(vr-1)$& odd & odd & odd & even & \checkmark & \checkmark\\
\hline
\end{tabular}
\caption{Case 3's integral solutions for $n = 6k+5$.}
\end{table}

    Case 4: $\begin{cases}u\mbox{ is odd}, u\ge 3, \mbox{ and } vr \equiv -1\mod u\mbox{, or }\\ u\mbox{ is even and }vr\equiv -1\mod 2u.\end{cases}$ 
    Observe that $u > r$ because $u = r$ implies that $u = 1$.

    \begin{table}[H]
\centering
\begin{tabular}{  |c| c| c| c| c| c|}
\hline
$r$& $u-r$ & $\frac{(u-r)v-1}{u}$ & $\frac{vr+1}{u}$ & $\Phi^{(0)}_1\ge 0?$ & $\Psi^{(0)}_1\ge 0?$ \\
\hline
$\ge 1$ & $\ge 1$ & $\ge 0$ & $\ge 1$ & \checkmark & \checkmark   \\
\hline
\end{tabular}
\caption{Case 4's nonnegative solutions for $n = 6k+5$.}
\end{table}

    \begin{table}[H]
\centering
\begin{tabular}{ |c| c| c| c| c| c| c|}
\hline
$n = 6k+5, k\ge 0, u\ge 2$ &$r$& $u-r$ & $\frac{(u-r)v-1}{u}$ & $\frac{vr+1}{u}$ & $\Phi^{(0)}_1\in\mathbb{Z}?$ & $\Psi^{(0)}_1\in \mathbb{Z}?$ \\
\hline
$2\nmid u$, $u|(vr+1)$, $2|v$ & even & odd & odd & odd & \checkmark & \checkmark   \\
\hline
$2\nmid u$, $u|(vr+1)$, $2\nmid v$ & odd & even & odd & even & \checkmark & \checkmark   \\
\hline
$2|u$, $(2u)|(vr+1)$& odd & odd & odd & even & \checkmark & \checkmark\\
\hline
\end{tabular}
\caption{Case 4's integral solutions for $n = 6k+5$.}
\end{table}
\end{proof}

\subsection{Application}
We use the theorems in Subsections \ref{neven} and \ref{nodd}  to find formulas for the solutions when $u = 1$ and $v\in \mathbb{N}$.

\begin{cor}
    Let $u = 1$ and $v$ is an odd positive integer. For $k\ge 0$, we have
    \begin{align*}
    &\frac{1}{2}\left(F_{6k+5}+(v-1)F_{6k+6}-1\right)t^{(1, v)}_{6k+6} + \frac{1}{2}\left(F_{6k+5}-1\right)t^{(1, v)}_{6k+7}\ =\ \frac{(t^{(1, v)}_{6k+6}-1)(t^{(1,v)}_{6k+7}-1)}{2};\\
    &\frac{1}{2}\left(F_{6k+1}-1\right)t^{(1,v)}_{6k+1} + \frac{1}{2}\left(F_{6k-1}+(v-1)F_{6k}-1\right)t^{(1,v)}_{6k+2}\ =\ \frac{(t^{(1,v)}_{6k+1}-1)(t^{(1,v)}_{6k+2}-1)}{2};\\
    &\frac{1}{2}\left(F_{6k+1} + (v-1)F_{6k+2}-1\right)t^{(1,v)}_{6k+2} + \frac{1}{2}\left(F_{6k+1}-1\right)t^{(1,v)}_{6k+3}\ =\ \frac{(t^{(1,v)}_{6k+2}-1)(t^{(1,v)}_{6k+3}-1)}{2};
    \end{align*}
    \begin{align*}
    1 + \frac{1}{2}&\left(F_{6k+2} + (v-1)F_{6k+3}-1\right)t^{(1,v)}_{6k+3} + \frac{1}{2}\left(F_{6k+2}-1\right)t^{(1,v)}_{6k+4}\\
    &\ =\ \frac{(t^{(1,v)}_{6k+3}-1)(t^{(1,v)}_{6k+4}-1)}{2};\\
    1 + \frac{1}{2}&\left(F_{6k+4} -1\right)t^{(1,v)}_{6k+4} + \frac{1}{2}\left(F_{6k+2}+(v-1)F_{6k+3}-1\right)t^{(1,v)}_{6k+5}\\
    &\ =\ \frac{(t^{(1,v)}_{6k+4}-1)(t^{(1,v)}_{6k+5}-1)}{2};\\
    1 + \frac{1}{2}&\left(F_{6k+4} + (v-1)F_{6k+5}-1\right)t^{(1,v)}_{6k+5} + \frac{1}{2}\left(F_{6k+4}-1\right)t^{(1,v)}_{6k+6}\\
    &\ =\ \frac{(t^{(1,v)}_{6k+5}-1)(t^{(1,v)}_{6k+6}-1)}{2}.
    \end{align*}
\end{cor}

\begin{cor}
    Let $u = 1$ and $v$ is an even positive integer. For $k\ge 0$, we have
    \begin{align*}
    &\frac{1}{2}\left(F_{6k+5} -1\right)t^{(1,v)}_{6k+5} + \frac{1}{2}\left(F_{6k+3}+(v-1)F_{6k+4}-1\right)t^{(1,v)}_{6k+6}\ =\ \frac{(t^{(1,v)}_{6k+5}-1)(t^{(1,v)}_{6k+6}-1)}{2};\\
    &\frac{1}{2}\left(F_{6k+5}+(v-1)F_{6k+6}-1\right)t^{(1, v)}_{6k+6} + \frac{1}{2}\left(F_{6k+5}-1\right)t^{(1, v)}_{6k+7}\ =\ \frac{(t^{(1, v)}_{6k+6}-1)(t^{(1,v)}_{6k+7}-1)}{2};\\
    &\frac{1}{2}\left(F_{6k+1}-1\right)t^{(1,v)}_{6k+1} + \frac{1}{2}\left(F_{6k-1}+(v-1)F_{6k}-1\right)t^{(1,v)}_{6k+2}\ =\ \frac{(t^{(1,v)}_{6k+1}-1)(t^{(1,v)}_{6k+2}-1)}{2};
    \end{align*}
    \begin{align*}
    1 + \frac{1}{2}&\left(F_{6k+2}-1\right)t^{(1,v)}_{6k+2} + \frac{1}{2}\left(F_{6k}+(v-1)F_{6k+1}-1\right)t^{(1,v)}_{6k+3}\\
    &\ =\ \frac{(t^{(1,v)}_{6k+2}-1)(t^{(1,v)}_{6k+3}-1)}{2};\\
    1 + \frac{1}{2}&\left(F_{6k+2} + (v-1)F_{6k+3}-1\right)t^{(1,v)}_{6k+3} + \frac{1}{2}\left(F_{6k+2}-1\right)t^{(1,v)}_{6k+4}\\
    &\ =\ \frac{(t^{(1,v)}_{6k+3}-1)(t^{(1,v)}_{6k+4}-1)}{2};\\
    1 + \frac{1}{2}&\left(F_{6k+4} -1\right)t^{(1,v)}_{6k+4} + \frac{1}{2}\left(F_{6k+2}+(v-1)F_{6k+3}-1\right)t^{(1,v)}_{6k+5}\\
    &\ =\ \frac{(t^{(1,v)}_{6k+4}-1)(t^{(1,v)}_{6k+5}-1)}{2};\\
    \end{align*}
\end{cor}

\begin{prob}\normalfont
    Find formulas of the solutions for more general sequences. In the case where consecutive terms of our sequence of interest are not necessarily relatively prime, we consider instead these terms divided by their greatest common divisor. 
\end{prob}

\section{Which equation to use} \label{whicheq}
We have looked at the solutions of \eqref{e1} and \eqref{e2} when $a$ and $b$ are consecutive terms of a given sequence, assuming that $\gcd(a,b) = 1$. However, we can expand our investigation to sequences whose consecutive terms are not necessarily relatively prime \cite{ACLLMM}. To do so, we define $\Gamma: \mathbb N^2\rightarrow\{0,1\}$ as follows: $\Gamma(a,b) = 0$ if 
$$\frac{a}{\gcd(a,b)}x + \frac{b}{\gcd(a,b)}y \ =\ \frac{1}{2}\left(\frac{a}{\gcd(a,b)}-1\right)\left(\frac{b}{\gcd(a,b)}-1\right)$$
has a nonnegative integral solution, and $\Gamma(a,b) = 1$ if 
$$1 + \frac{a}{\gcd(a,b)}x + \frac{b}{\gcd(a,b)}y \ =\ \frac{1}{2}\left(\frac{a}{\gcd(a,b)}-1\right)\left(\frac{b}{\gcd(a,b)}-1\right)$$
has a nonnegative integral solution.

\begin{prob}\normalfont\label{prob2}
    Given a sequence $(a_n)_{n=1}^\infty$, what is the sequence $(\Gamma(a_n, a_{n+1}))_{n=1}^\infty$?
\end{prob}

For example, if we have a geometric progression $(a_n: = ar^{n-1})_{n=1}^\infty$, then 
$$\Gamma(a_n, a_{n+1})\ =\ \Gamma(1, r)\ = \ 0, \mbox{ for all }n\in\mathbb{N}.$$
Hence, the sequence $\Delta((a_n)_{n=1}^\infty) = 1, 1, 1, \ldots$. 
On the other hand, if $b_1 = b$ for some $b\ge 2$, and $b_n = 2b_{n-1} - 1$ for each $n\ge 2$, then
$$\Gamma(b_n, b_{n+1})\ =\ \Gamma(b_n, 2b_{n}-1)\ =\ 2, \mbox{ for all }n\in \mathbb{N},$$
because 
$$1  + b_n\cdot (b_n-2) + (2b_n-1)\cdot 0\ =\ \frac{(b_n-1)(2b_n-2)}{2}.$$
In this case, we have $\Delta((b_n)_{n=1}^\infty) = 2, 2, 2, \ldots$.

This section presents selected results from \cite{ACLLMM}, including a theorem to compute $\Gamma(a,b)$ and its application in solving Problem \ref{prob2} for various sequences. 

Given $(a,b)\in \mathbb{N}^2$ with $b/\gcd(a,b) > 1$, define $\Theta (a,b)$ to be the unique multiplicative inverse of $a/\gcd(a,b)$ in modulo $b/\gcd(a,b)$ such that $0<\Theta (a,b)<b/\gcd(a,b)$. 

\begin{thm}\cite[Theorem 1.1]{ACLLMM}\label{thm1.1} Let $a, b\in \mathbb{N}$. If $a$ divides $b$ or $b$ divides $a$, then $\Gamma(a,b) = 0$. Otherwise, the following hold. 
    \begin{enumerate}
        \item [a)] When $a/\gcd(a,b)$ is odd, then $\Gamma(a,b) = 0$ if and only if $\Theta(b,a)$ is odd.
        \item [b)] When $a/\gcd(a,b)$ is even, then $\Gamma(a,b) = 0$ if and only if $\Theta(a,b)$ is odd.
    \end{enumerate}
\end{thm}

\begin{proof}
  The first statement follows from 
  $$\Gamma(a,b)\ =\ \begin{cases} \Gamma(1, b/a), &\mbox{ if }a|b;\\ \Gamma(a/b, 1), &\mbox{ if }b|a;\end{cases}\ =\ 0.$$ 

  Suppose that $a$ does not divide $b$, $b$ does not divide $a$, and $a/\gcd(a,b)$ is odd. Let $A = a/\gcd(a,b)$ and $B = b/\gcd(a,b)$. 
  
    If $\Gamma (a,b)=0$, then $Ax + By = (A-1)(B-1)/2$; equivalently, 
    $$2Ax+2By\ =\ AB-A-B+1.$$ 
    Hence, 
    $$(2y+1)B \ \equiv\ 1 \mod A.$$ 
    Since $\Theta (b,a)B \equiv 1$ mod $A$ and $\gcd(A, B) = 1$, we have 
    $$2y+1 \ \equiv\ \Theta (b,a) \mod A.$$
    Observe that
    $$
        0 \ <\ 2y + 1\ =\ \frac{AB-A-2Ax+1}{B}\ < \ A.
    $$
    Therefore, $2y+1=\Theta (b,a)$ and so, $\Theta(b,a)$ is odd.
    
    If $\Gamma (a,b)=1$, then  $1+Ax + By = (A-1)(B-1)/2$; equivalently,
    $$2Ax+2By\ =\ AB-A-B-1.$$ 
    Hence, 
    $$-(2y+1)B \ \equiv\ 1\mod A.$$
    Since $\Theta (b,a)B \equiv 1 \mod A$ and $\gcd(A, B) = 1$, we have 
    $$A-(2y+1)\ \equiv\ \Theta(b,a)\mod A.$$
    Observe that $A-(2y+1) < A$ and
    $$A-(2y+1) \ =\ A - \frac{AB-2Ax-A-1}{B}\ =\ \frac{2Ax+A+1}{B}\ > \ 0.$$
    Therefore, $A-(2y+1)  = \Theta(b, a)$. That $A$ is odd implies that $\Theta(b,a)$ is even. 
    
    We have shown that if $A$ is odd, $\Gamma(a,b) = 0$ if and only if $\Theta(b,a)$ is odd. 

    It remains to show that if $A$ is even, then $\Gamma(a,b) = 0$ if and only if $\Theta(a, b)$ is odd. However, this is obvious from the fact that when $A$ is even, we have odd $B$. By above, $\Gamma(a,b) = 0$ if and only if $\Theta(a,b)$ is odd. 
\end{proof}

\begin{rek}\normalfont
In the proof of Theorem \ref{thm1.1} for the case $a\nmid b$, $b\nmid a$, and $a/\gcd(a,b)$ is even, we need only $b/\gcd(a,b)$ is odd, which is guaranteed by $a/\gcd(a,b)$ is even. Therefore, the second statement of Theorem \ref{thm1.1} can be restated as follows: suppose that $a\nmid b$ and $b\nmid a$. The following hold.
    \begin{enumerate}
        \item [a)] When $a/\gcd(a,b)$ is odd, then $\Gamma(a,b) = 0$ if and only if $\Theta(b,a)$ is odd.
        \item [b)] When $b/\gcd(a,b)$ is odd, then $\Gamma(a,b) = 0$ if and only if $\Theta(a,b)$ is odd.
    \end{enumerate}
\end{rek}

Next, we apply Theorem \ref{thm1.1} to different sequences $(a_n)_{n=1}^\infty$ and determine the equation used by consecutive terms of $(a_n)_{n=1}^\infty$. Let $\Delta((a_n)_{n=1}^\infty) := (\Gamma(a_n, a_{n+1}))_{n=1}^\infty$.

\begin{thm}\cite[Theorem 1.5]{ACLLMM}\label{thm1.5}
For each $k \in \mathbb{N}$, the sequence $\Delta((n^k)_n)$ is eventually $0,1,0,1,0,1,\dots$. 
\end{thm}
\begin{proof}
    Suppose that $k$ is odd. For $n \in \mathbb{N}_{>2}$, let $s = \Theta((n-1)^k, n^k)$ and $t = \Theta((n+1)^k, n^k)$. Since $k$ is odd, we can write
    $$-s \ =\ (n^k-1)^k s\ =\ \left(\sum_{i=1}^{k} n^{i-1}\right)^k (n-1)^k s\mod n^k.$$
    Using $(n-1)^k s \equiv 1 \mod{n^k}$, we obtain 
    \begin{equation} \label{1.5.1}
        s \ \equiv\ n^k - \left( \sum_{i=1}^k n^{i-1}\right)^k \mod n^k.
    \end{equation}
    Similarly, 
    \begin{equation} \label{1.5.2}
        t \ \equiv\ \left( \sum_{i=1}^k (-1)^{i-1} n^{i-1}\right)^k \mod n^k.
    \end{equation}
    
    Let $u(X) := \left( \sum_{i=1}^k X^{i-1}\right)^k$ and $v(X) := \left( \sum_{i=1}^k (-1)^{i-1} X^{i-1}\right)^k$. Since $(u+v)(X)$ is an even polynomial, the coefficient of each odd power in $v(X)$ is  the negative of the corresponding coefficient of $u(X)$. On the other hand, $(u-v)(X)$ is an odd polynomial, the coefficient of each even power in $v(X)$ equals the corresponding coefficient of $u(X)$. It follows that if 
    $g(X)$ is the tail of $u(X)$ up to the power $k-1$, then $g(-X)$ is the tail of $v(X)$ up to the power $k-1$. By \eqref{1.5.1} and \eqref{1.5.2}, 
    \begin{equation*}
        s \ \equiv\ n^k - g(n) \quad \mbox{ and } \quad t \ \equiv\ g(-n) \mod{n^k}.
    \end{equation*}
    Choose $N \in \mathbb{N}$ such that for all $n \ge N$, 
    $$n^k \ >\ g(n)\ \ge \ g(-n), \quad  g(n) \ >\ 0,\quad \mbox{ and }\quad g(-n) > 0,$$
    which is possible due to odd $k$.
    It follows that
    \begin{equation} \label{1.5.4}
        s \ =\ n^k - g(n) \quad \mbox{ and } \quad t \ =\ g(-n), \mbox{ for all } n \ge N.
    \end{equation}
    Since all coefficients of $g(x) - g(-x)$ are even, $s+t = n^k - (g(n)-g(-n))$ has the same parity as $n^k$. 
    
    Take an even $M \ge N-1$. By above, $\Theta(M^k, (M+1)^k) + \Theta ((M+2)^k, (M+1)^k)$ has the same parity as $(M+1)^k$, which is odd. Hence $\Theta(M^k, (M+1)^k)$ and $\Theta ((M+2)^k, (M+1)^k)$ have different parities. By Theorem \ref{thm1.1},
    \begin{equation*}
        \Gamma(M^k, (M+1)^k) \ \neq\ \Gamma((M+1)^k, (M+2)^k).
    \end{equation*}
    Here we use the assumption that $M$ is even. To finish the proof that $\Delta((n^k)_{n=1}^\infty)$ eventually alternates between $0$ and $1$, it remains to verify that $\Gamma((M+1)^k, (M+2)^k)\neq \Gamma((M+2)^k, (M+3)^k)$ or equivalently, $\Gamma(M^k, (M+1)^k) = \Gamma((M+2)^k, (M+3)^k)$. By \eqref{1.5.4},
    \begin{align*}
        &\Theta(M^k, (M+1)^k)\ =\ (M+1)^k - g(M+1),\mbox{ and }\\
        &\Theta((M+2)^k, (M+3)^k)\ =\ (M+3)^k - g(M+3).
    \end{align*}
    Therefore,
    $\Theta((M+2)^k, (M+3)^k) - \Theta(M^k, (M+1)^k) = (M+3)^k - (M+1)^k - (g(M+3)-g(M+1))$, which is even; that is, $\Theta((M+2)^k, (M+3)^k)$ and $\Theta(M^k, (M+1)^k)$ have the same parity. By Theorem \ref{thm1.1}, $\Gamma(M^k, (M+1)^k) = \Gamma((M+2)^k, (M+3)^k)$.

   The proof for even $k$ is similar and is left for interested readers, who may also find the proof in \cite[Section 3]{ACLLMM}.    
\end{proof}

\begin{rek}\normalfont
  The readers may refer to \cite[Theorem 1.5]{ACLLMM} for an upper bound of when the alternating pattern starts. 
\end{rek}

\begin{thm}\cite[Theorem 1.6]{ACLLMM}\label{thm1.6}
Let $(a_{n})_{n\ge 1}$ be an arithmetic progression: $a_n = a + (n-1)r$ with $a, r\in \mathbb{N}$. Then $\Delta((a_{n})_{n})$ is either $1,0,1,0,\ldots$ or $0,1,0,1\ldots$. 
\end{thm}
\begin{proof}
Let $n \in \mathbb{N}$. We have that $\gcd(a_n, a_{n+1})$ divides $(2a_{n+1} - a_n)$, which is $a_{n+2}$, so
$$\gcd(a_n, a_{n+1})\,|\,\gcd(a_{n+1}, a_{n+2}).$$
Conversely, $\gcd(a_{n+1}, a_{n+2})$ divides $(2a_{n+1} - a_{n+2})$, which is $a_n$, 
$$\gcd(a_{n+1}, a_{n+2})\,|\, \gcd(a_n, a_{n+1}).$$ Hence, $\gcd(a_n, a_{n+1}) = \gcd(a_{n+1}, a_{n+2})$. 
Furthermore, for $n\in \mathbb{N}$, writing  $a_n = a + r(n-1)$ gives
$$\gcd(a_n, a_{n+1})\ =\ \gcd(a + r(n-1), a + rn)\ =\ \gcd(a+r(n-1), r) \ =\ \gcd(a,r).$$
Therefore, we can set $d := \gcd(a_n, a_{n+1})$ for all $n\in \mathbb{N}$. 

We need to show that
$$\Gamma(a_n, a_{n+1}) \ \neq\ \Gamma(a_{n+1}, a_{n+2}), \mbox{ for all }n\in\mathbb{N}.$$
We first suppose that $n\ge 2$ to take advantage of the fact that $a_n\nmid a_{n+1}$.

Case 1: $a_{n+1}/d$ is odd. Let $x = \theta(a_n, a_{n+1})$ and $y = \theta(a_{n+2}, a_{n+1})$. Then
$$\frac{a_n}{d} x \ \equiv\ 1 \quad \mbox{ and }\quad\left(\frac{a_n}{d} + \frac{2r}{d}\right) y\ \equiv\ 1 \mod \left(\frac{a_n}{d} + \frac{r}{d}\right).$$
Hence,
$$
\frac{r}{d}\left(\frac{a_n}{d} + \frac{r}{d} - x\right)\ \equiv\ 1 \quad \mbox{ and }\quad \frac{r}{d} y\ \equiv\ 1\mod \left(\frac{a_n}{d} + \frac{r}{d}\right).
$$
Furthermore, since $1 \le x < a_{n+1}/d$, we know that 
$$0\ <\ \frac{a_{n+1}}{d} - x, y  \ <\ \frac{a_{n+1}}{d}.$$
It follows that $a_{n+1}/d - x = y$, so $x + y = a_{n+1}/d$, which is odd. Hence, $x$ and $y$ have different parities. By Theorem \ref{thm1.1}, $\Gamma(a_n, a_{n+1}) \neq \Gamma(a_{n+1}, a_{n+2})$.

Case 2: $a_{n+1}/d$ is even. Then $a_n/d$ and $(a_n+2r)/d$ are odd. Let $x = \theta(a_{n+1}, a_n)$ and $y = \theta(a_{n+1}, a_{n+2})$. Then
$$\left(\frac{a_n}{d} + \frac{r}{d}\right) x \ \equiv\ 1 \mod\frac{a_n}{d}\quad\mbox{ and }\quad\left(\frac{a_n}{d} + \frac{r}{d}\right) y \ \equiv\ 1 \mod \left(\frac{a_n}{d} + \frac{2r}{d}\right).$$
Equivalently, there exist positive integers $k_1, k_2 < (a_n + r)/d$ such that
$$
\left(\frac{a_n}{d} + \frac{r}{d}\right) x \ =\ 1 + k_1 \frac{a_n}{d}\quad\mbox{ and }\quad\left(\frac{a_n}{d} + \frac{r}{d}\right) y\ =\ 1 + k_2 \left(\frac{a_n}{d} + \frac{2r}{d}\right),
$$
which gives
$$\frac{a_n + r}{d}(x - y + 2k_2) \ =\ \frac{a_n}{d}(k_1 + k_2).$$
Since $\gcd((a_n + r)/d, a_n/d) = 1$, $(a_n + r)/d$ divides $k_1 + k_2$. Moreover, $0 < k_1 + k_2 < 2(a_n + r)/d$ because $0 < k_1, k_2 < (a_n + r)/d$, so
$$k_1 + k_2 \ =\ \frac{a_n + r}{d} \ \Longrightarrow\ x - y + 2k_2 \ =\ \frac{a_n}{d}.$$
Odd $a_n/d$ implies that $x$ and $y$ must have different parities. Hence, $\Gamma(a_n, a_{n+1}) \neq \Gamma(a_{n+1}, a_{n+2})$.

It remains to show that $\Gamma(a_1, a_2) \neq \Gamma(a_2, a_3)$. If $a_1 \nmid a_2$, the same reasoning as above gives $\Gamma(a_1, a_2) \neq \Gamma(a_2, a_3)$. If $a_1 \mid a_2$, $\Gamma(a_1, a_2) = 0$. Let $a_2 = pa_1 = pa$ for some $p \ge 2$. Then
$$a_3 \ =\ 2a_2 - a_1 \ =\ (2p - 1)a \quad \text{and}\quad \Gamma(a_2, a_3) \ =\ \Gamma(pa, (2p - 1)a)\ =\ \Gamma(p, 2p - 1).$$
We have
$$1 + p\cdot (p-2) + (2p-1)\cdot 0\ =\ \frac{(p-1)(2p-2)}{2},$$
so $\Gamma(a_2, a_3) = 1 \neq \Gamma(a_1, a_2)$.
\end{proof}

\begin{prob}\normalfont
Let $\mathcal{F} = \{(a_n)_{n=1}^\infty: \Delta((a_n)_n)\mbox{ eventually alternates between }0\mbox{ and }1\}$. Characterize sequences that belong to $\mathcal{F}$.  According to Theorems \ref{thm1.5} and \ref{thm1.6}, $(n^k)_{n=1}^\infty$ and arithmetic progressions of positive integers are in $\mathcal{F}$. 
\end{prob}

The next result studies the behavior of $\Gamma$ when we fix one parameter and let the other vary.

\begin{thm}\cite[Theorem 1.8]{ACLLMM}
    Let $k \in \mathbb{N}$. The following holds.
    \begin{enumerate}
    \item If $k$ is odd, $(\Gamma(k, n))_{n=1}^\infty$ has period $k$. In each period, the number of $0$'s is one more than the number of $1$'s.
    \item If $k$ is even, $(\Gamma(k,n))_{n=1}^\infty$ has period $2k$. In each period, the number of $0$'s is two more than the number of $1$'s.
    \end{enumerate}
\end{thm}

\begin{proof}[Proof for odd $k$]

If $k = 1$, we have $\Gamma(1, n) = 0$ for every $n\in \mathbb{N}$, so the statement holds for $k = 1$.
Assume that $k\ge 3$.\newline

Step 1: The sequence $(\Gamma(k,n))_{n=1}^\infty$ is periodic, and its period divides $k$. 

Let $u, v\in \mathbb{N}$ with $v = u + k$. We need to prove that $\Gamma(k, u) = \Gamma(k, v)$. We proceed by case analysis. 

\begin{enumerate}
\item[a)] Case 1: $k$ divides $u$. Then $k$ also divides $v$, so $\Gamma (k,u) = \Gamma (k,v) = 0$.
\item[b)] Case 2: $u$ divides $k$. Then $\Gamma (k, u) = 0$. Write $k = u\ell$ for some odd $\ell\in \mathbb{N}$. We have
$$\Gamma(k, v) \ =\ \Gamma(u\ell, u(\ell+1))\ =\ \Gamma(\ell, \ell+1)\ =\ 0,$$
because
$$\ell\cdot \frac{\ell-1}{2} + (\ell+1)\cdot 0\ =\ \frac{(\ell-1)\ell}{2}.$$
\item[c)] Case 3: $u$ does not divide $k$, and $k$ does not divide $u$. Let $d = \gcd(k,u)$. Since $k$ is odd, $k/d$ is odd. By Theorem \ref{thm1.1}, it suffices to show that $\Theta(u,k)$ and $\Theta(v,k)$ have the same parity. We have 
$$\frac{u\Theta(u,k)}{d} \ \equiv\ 1 \quad \mbox{ and }\quad \frac{v\Theta(v,k)}{d} \equiv 1 \mod \frac{k}{d}.$$
Since $v \equiv u \mod k$, we also have
$$\frac{u\Theta(v,k)}{d} \ \equiv\ 1 \mod \frac{k}{d}.$$ 
Hence, $\Theta(u,k) = \Theta(v,k)$.\newline
\end{enumerate}

Step 2: In $(\Gamma(k,n))_{n=1}^k$, the number of $0$'s is one more than the number of $1$'s. 

Pick $1 \le s \le (k-1)/2$ and let $r = \gcd(k,s) = \gcd(k-s,k)$. 

\begin{enumerate}
\item[a)] Case 1: If $s$ divides $k$, then $\Gamma(k,s) =0$. Write $k =s\ell$ for some odd $\ell \in \mathbb{N}_{\ge 3}$. We have $\Gamma(k, k-s) = \Gamma(\ell, \ell-1)=1$ because
$$1 + \ell\cdot \frac{\ell-3}{2} + (\ell-1)\cdot 0 \ =\ \frac{(\ell-1)(\ell-2)}{2}.$$
Hence, $\Gamma(k, k-s)\neq \Gamma(k,s)$.
\item[b)] Case 2: $s$ does not divide $k$. Observe that $k/2 < k-s <k$, so $k-s$ does not divide $k$. By Theorem \ref{thm1.1}, $\Gamma(k,s)$ and $\Gamma(k, k-s)$ are determined by the parity of $\Theta(s,k)$ and $\Theta(k-s,k)$, respectively. We have 
$$\Theta(s,k)\frac{s}{r} \ \equiv\ 1 \quad\mbox{ and }\quad\Theta(k-s,k)\frac{k-s}{r} \ \equiv\ 1 \mod \frac{k}{r}.$$
Hence,
$$\left( \frac{k}{r}-\Theta(k-s,k)\right)\frac{s}{r} \ \equiv\ 1 \mod \frac{k}{r}.$$ 
It follows that 
$$\Theta(s,r) +\Theta(k-s,k)\ =\ \frac{k}{r}.$$ Since $k/r$ is odd, we have $\Theta(s,k)\not\equiv\Theta(k-s,k)\mod2$. Therefore, $\Gamma(k,s)\neq\Gamma(k,k-s)$. 
\end{enumerate}

We have shown that for all $1 \le s \le (k-1)/2$, it holds that $\Gamma(k, s)\neq \Gamma(k, k-s)$. Along with  $\Gamma(k,k) = 0$, we conclude that for the $k$ terms  $(\Gamma(k,n))_{n = 1}^k$, the number of $0$'s is one more than the number of $1$'s.\newline

Step 3: $(\Gamma(k,n))_{n = 1}^\infty$ has period $k$. 

Let $T$ be the period of $(\Gamma(k,n))_{n = 1}^\infty$. By Step 1,  $T$ divides $k$. Hence, within the first $k$ terms, there are $k/T$ copies of the period. Let $p$ and $q$ be the number of $0$'s and $1$'s in each period. Then $(p-q)(k/T) =1$, which implies that $p-q = k/T =1$. Therefore, $(\Gamma(k,n))_{n = 1}^\infty$ has period $k$.
\end{proof}
    
\begin{prob}\normalfont\cite[Section 6]{ACLLMM} Let $H(x)$ be the density of all pairs $(a,b)\in \mathbb{N}^2$ with $1\le a\le b\le x$ and $\Gamma(a,b)=0$, i.e., 
    $$H(x)\ :=\ \frac{\#\{(a,b)\in \mathbb N^2\,:\, 1\le a\le b\le x, \Gamma(a, b) = 0\}}{\#\{(a,b)\in \mathbb N^2\,:\, 1\le a\le b\le x\}}.$$
Adding the condition $\gcd(a,b) = 1$, we obtain  
$$G(x)\ :=\ \frac{\#\{(a,b)\in \mathbb N^2\,:\, 1\le a\le b\le x, \Gamma(a, b) = 0, \gcd(a,b) = 1\}}{\#\{(a,b)\in \mathbb N^2\,:\, 1\le a\le b\le x, \gcd(a,b) = 1\}}.$$
Prove what the data in \cite[Section 6]{ACLLMM} suggest, i.e., 
$$\lim_{x\rightarrow\infty} G(x) = 0.5\ \neq \ \lim_{x\rightarrow\infty} H(x) \approx 0.304\ldots.$$
\end{prob}

\end{document}